\documentclass{tac}

\usepackage{amsmath,amsfonts,mathabx,tikz,tikz-cd,enumitem,stmaryrd,url,mathtools}

\usetikzlibrary{matrix,arrows,decorations.pathmorphing}

\usepackage[colorlinks=true]{hyperref}
\hypersetup{allcolors=[rgb]{0.1,0.1,0.4}}

\author{Branko Nikoli\'c}

\thanks{The author gratefully acknowledges the support of an International Macquarie University Research Scholarship.}

\address{Mathematics department, Macquarie University, \\ NSW 2109, Australia}

\title {Strictification tensor product of 2-categories}

\copyrightyear{2018}


\eaddress{branko.nikolic@mq.edu.au}

\theoremstyle{definition}
\newtheorem{defn}{Definition}[section]
\newtheorem{prop}{Proposition}[section]
\newtheorem{thm}{Theorem}[section]

\newtheorem{cor}{Corollary}[section]

\makeatletter
\newcommand{\ptoLR}{}
\newcommand{\pgetsLR}{}

\DeclareRobustCommand{\ptoLR}{\mathrel{\mathpalette\p@toLR@getsLR\to}}
\DeclareRobustCommand{\pgetsLR}{\mathrel{\mathpalette\p@toLR@getsLR\gets}}

\newcommand{\p@toLR@getsLR}[2]{
  \ooalign{\hidewidth
  $\m@th#1\mapstochar
  \mkern2.5mu
  $
  \hidewidth\cr
  $\m@th#1\longleftrightarrow$
  \cr
  }%
}
\makeatother

\begin{document}

\maketitle
\begin{abstract}
Given 2-categories $\mathcal{C}$ and $\mathcal{D}$, let $\textrm{Lax}(\mathcal{C},\mathcal{D})$ denote the 2-category of lax functors, lax natural transformations and modifications, and  $[\mathcal{C},\mathcal{D}]_\mathrm{lnt}$ its full sub-2-category of (strict) 2-functors. We give two isomorphic constructions of a 2-category $\mathcal{C}\boxtimes\mathcal{D}$ satisfying $\textrm{Lax}(\mathcal{C},\textrm{Lax}(\mathcal{D},\mathcal{E})) \cong [\mathcal{C}\boxtimes \mathcal{D},\mathcal{E}]_\mathrm{lnt}$, hence generalising the case of the free distributive law $1\boxtimes 1$. We also discuss dual constructions.
\end{abstract}

\tableofcontents

\section{Introduction}

Monads (aka triples, standard constructions) are given by a category $C$, an endofunctor $F:C\rightarrow C$ and two natural transformations $\eta:1_C\Rightarrow F$ and $\mu:F^2\Rightarrow F$, satisfying unit and associativity axioms \cite{Lane1998}. Their use is ubiquitous and the most common one is describing a (possibly complicated) algebraic structure as Eilenberg-Moore (EM) algebras \cite{Lane1998} on a category of simpler ones. An EM algebra is given by a map $TX\rightarrow X$ compatible with $\mu$ and $\eta$. With algebra morphisms, they form a category $\mathrm{EM}(T)$. The full subcategory of $\mathrm{EM}(T)$ consisting of free algebras is (up to equivalence) usually denoted $\mathrm{KL}(T)$. A typical example is the Abelian group monad on the category of sets taking a set $S$ to the set of elements of the free Abelian group on $S$. 

A distributive law \cite{Beck1969} consists of two different monads on the same category satisfying a compatibility condition. Then their composite is a new monad. A typical example is the Abelian group monad together with the monoid monad producing the ring monad, hence the name.

Monads are in fact definable in an arbitrary bicategory $\mathcal{E}$ \cite{Street1972}, just by replacing words ``functor'' with arrow and natural transformation by 2-cell. For example, in a bicategory of spans, monads are precisely (small) categories \cite{Benabou1967}. A morphism between a monad $T$ on $X$ and $S$ on $Y$, consists of an arrow $X\xrightarrow{F}Y$ and a ``crossing'' 2-cell $S\circ F\xRightarrow{\sigma}F\circ T$ which is compatible with unit and multiplication for both monads. A morphism between monad morphisms $F$ and $G$, consists of a 2-cell $F\xRightarrow{\alpha}G$ compatible with crossing 2-cells. These form the 2-category of monads in $\mathcal{E}$, called $\mathrm{Mnd}(\mathcal{E})$. Now, a distributive law in $\mathcal{E}$ has a short description as a monad in $\mathrm{Mnd}(\mathcal{E})$. Various duals are expressible using dualities of 2-categories, for instance, the 2-category of comonads is defined as $\mathrm{Cmd}(\mathcal{E})=\mathrm{Mnd}(\mathcal{E}^\mathrm{co})^\mathrm{co}$, mixed distributive laws as $\mathrm{Cmd}(\mathrm{Mnd}(\mathcal{E}))$. Since objects of $\mathcal{E}$ are no longer categories, we have no access to their elements, and cannot form an $EM$-category; but we can use the 2-dimensional universal property of lax limit to obtain, if exists, an $\mathrm{EM}$-object $\mathrm{EM}(T)$, also denoted $C^T$. The main topic of \cite{Lack2002} is completion of $\mathcal{E}$ under these limits. Dually, lax colimits give $\mathrm{KL(T)}$, also denoted $C_T$.

The free monad \cite{Lawvere1969} is a 2-category $FM$ which classifies monads; that is, the 2-category of strict functors, lax natural transformations and modifications $[\mathrm{FM},\mathcal{E}]_\mathrm{lnt}$ is isomorphic to $\mathrm{Mnd}(\mathcal{E})$. It is given by the suspension of the opposite of the algebraist's category of simplices, $\Delta_\mathrm{a}^\mathrm{op}$ with ordinal sum as the monoidal structure. We will use it a lot, so we review its definition and some properties in Appendix \ref{sec:simpaths}. The free mixed distributive law ($\mathrm{FMDL}$) was constructed by Street \cite{Street2015}, and is a special case of the construction presented here.

A lax functor \cite{Benabou1967} (aka morphism) between bicategories generalises the notion of a (strict) 2-functor, by relaxing the conditions of preservation of the unit and composition of arrows. Instead, a lax functor $F:\mathcal{D}\rightarrow\mathcal{E}$ is equipped with comparison maps 
\begin{align*}
\eta_D:1_{FD}\Rightarrow F(1_D)\textrm{ and }\mu_{dd'}:F(d')\circ F(d)\Rightarrow F(d'\circ d)
\end{align*}
for each object $D$ of $\mathcal{D}$, and composable pair $(d,d')$ of arrows in $\mathcal{D}$. These are required to satisfy unit and associativity laws, and $\mu$ is required to be natural in $c$ and $c'$. The special case of $\mathcal{D}=1$, that is, if $\mathcal{D}$ has only one 0/1/2-cell, then giving a lax functor exactly corresponds to giving a monad in $\mathcal{E}$. A lax functor from the chaotic category\footnote{That is, the category having exactly one arrow in each hom.}
on a set $X$ corresponds to a category enriched in $\mathcal{E}$. Another example, lax functors from $\mathbb{I}(:=0\rightarrow 1)$ into $\mathrm{Span}$ correspond to choosing two categories and a module (aka profunctor, distributor) between them. Lax natural transformations $F\xRightarrow{\sigma}G$ between two such functors consist of arrows $FD\xrightarrow{\sigma_D}GD$, for each $D\in \mathcal{D}$, and $Gd\circ \sigma_D\xrightarrow{\sigma_d}\sigma_D'\circ Fd$, for each $D\xrightarrow{d}D'$ in $\mathcal{D}$, natural in $d$ and compatible with $\eta$ and $\mu$. Finally a modification $\sigma\xrightarrow{m}\tau$  consists of 2-cells $\sigma_D\xRightarrow{m_D}\tau_D$, for each $D$, compatible with $\sigma$. These form a 2-category $\mathrm{Lax}(\mathcal{D},\mathcal{E})$. The choice of directions gives an isomorphism of 2-categories $\mathrm{Lax}(1,\mathcal{E})\cong \mathrm{Mnd}(\mathcal{E})$, and by the definition of (free) distributive law ($\mathrm{FDL}$) we have $\mathrm{Lax}(1,\mathrm{Lax}(1,\mathcal{E}))\cong [\mathrm{FDL},\mathcal{E}]_\mathrm{lnt}$.

Our goal is, given 2-categories $\mathcal{C}$ and $\mathcal{D}$, to construct a 2-category $\mathcal{C} \boxtimes\mathcal{D}$ that is ``free'', in the sense that it strictifies the lax functors, so that
\begin{equation}\label{eq:mainiso}
\mathrm{Lax}(\mathcal{C},\mathrm{Lax}(\mathcal{D},\mathcal{E}))\cong [\mathcal{C} \boxtimes\mathcal{D},\mathcal{E}]_\mathrm{lnt}.
\end{equation}

The variables $C$, $c$, $\gamma$ used to describe cells in $\mathcal{C}$ (similarly for $D$, $d$ and $\delta$ in $\mathcal{D}$), have sources and targets according to the diagram
\ref{diag:notation}.
\begin{equation}\label{diag:notation}
\begin{tikzpicture}[
baseline=(current  bounding  box.center)]
\def \str {2.0}
\def \gmoff {0.14}
\node (C1) at (\str*1,0) {$C$};
\node (C2) at (\str*2.5,0) {$C'$};
\node (C3) at (\str*4,0) {$C''$};
\path[->,font=\scriptsize,>=angle 90]
(C1) edge [bend left] node[above] {$c$} (C2);
\path[->,font=\scriptsize,>=angle 90]
(C1) edge node[above right] {$\bar c$} (C2);
\path[->,font=\scriptsize,>=angle 90]
(C1) edge [bend right] node[below] {$\bar{\bar c}$} (C2);
\path[->,font=\scriptsize,>=angle 90]
(C2) edge [bend left] node[above] {$c'$} (C3);
\path[->,font=\scriptsize,>=angle 90]
(C2) edge node[above right] {$\bar c'$} (C3);
\path[->,font=\scriptsize,>=angle 90] 
(C2) edge [bend right] node[below] {$\bar{\bar c}'$} (C3);
\node at (\str*1.6,\str * \gmoff) {\scriptsize $\gamma$ {$\Downarrow$}};
\node at (\str*1.7,-\str * \gmoff) {\scriptsize $\bar\gamma$ {$\Downarrow$}};
\node at (\str*3.1,\str * \gmoff) {\scriptsize $\gamma'$ {$\Downarrow$}};
\node at (\str*3.2,-\str * \gmoff) {\scriptsize $\bar \gamma'$ {$\Downarrow$}};
\end{tikzpicture}
\end{equation}
Horizontal composition is denoted by $\circ$ and vertical by $\bullet$.

\section{Tensor product via computads}

We begin by fully unpacking the LHS of (\ref{eq:mainiso}), which involves familiar, but numerous axioms - there are eighteen axioms for an object (lax functor) $B$, five axioms for an arrow (lax natural transformation) $b:B\rightarrow B'$, and two axioms for a 2-cell (modification) $\beta:b\Rightarrow \bar b$. Then we review the definition of computads \cite{Street1976} which play the same role for 2-categories as graphs do for usual categories - they are part of a monadic adjunction. We then proceed to construct a computad $\mathcal{G}$ to give a convenient generator-relation description of the tensor product. 

\subsection{Unpacking} \label{sec:unpacking}
An object $B$ of $\textrm{Lax}(\mathcal{C},\textrm{Lax}(\mathcal{D},\mathcal{E}))$ assigns to each $C\in \mathcal{C}$ a lax functor $B C:\mathcal{D}\rightarrow \mathcal{E}$, which amounts to giving the following data in $\mathcal{E}$: 
\begin{itemize}
	\item for each $D$ an object $B CD\in \mathcal{E}$
	\item for each $d$ an arrow $B Cd:B CD\rightarrow B CD'$
	\item for each $\delta$ a 2-cell $B C\delta:B Cd\Rightarrow B C\bar d$, functorially
	\begin{align}
	BC1_d &= 1_{BCd}			\label{eq:functid}\\
	BC(\bar\delta\bullet \delta) &= BC\bar\delta \bullet BC\delta
	\end{align}
	\item{\hspace{-0.2cm}\tiny(\textbf{f1})} for each $D$ a unit comparison 2-cell $\eta_{BC1_D}:1_{B CD}\Rightarrow B C1_D$
	\item{\hspace{-0.2cm}\tiny(\textbf{f1})} for each composable pair $(d,d')$ a composition comparison 2-cell 
	$\mu_{BCdd'}:(B Cd')\circ (B C d)\Rightarrow (B Cd'\circ d)$,\newline
satisfying unit and associativity axioms,
\begin{align}
\mu
\bullet (1
\circ  \eta
) = &\,\,1
 = \mu
\bullet (\eta%
\circ  1
)\label{eq:unit1}\\
\mu\bullet (1 \circ  \mu) &= \mu \bullet (\mu \circ  1)
\label{eq:comp1}
\end{align}
together with a naturality condition,
\begin{equation}
\mu_{BC\bar d\bar d'}\bullet (BC\delta' \circ  BC\delta) = BC(\delta'\circ  \delta)\bullet\mu_{Cdd'}.
\end{equation}
\end{itemize}
Also, $B$ assigns to each $c:C\rightarrow C'$ a lax natural transformation
$Bc:BC\rightarrow B C'$ consisting of:
\begin{itemize}
	\item arrows $B cD: B CD\rightarrow B C'D$
	\item{\hspace{-0.2cm}\tiny(\textbf{t1})} 2-cells $\sigma_{B cd}: B C'd\circ B c D \Rightarrow B cD'\circ B C d$,\newline
	with the two axioms expressing compatibility with unit and composition, 
	\begin{align}
	\sigma\bullet (\eta \circ  1) &= 1\circ  \eta\label{eq:swapunit1}\\
	\sigma\bullet (\mu \circ  1) &= (1 \circ  \mu) \bullet (\sigma\circ  1)\bullet (1\circ  \sigma)\label{eq:swapcomp1}
	\end{align}
	and one expressing naturality, 
	\begin{equation}
	\sigma_{Bc\bar d}\bullet(BC'\delta\circ  1_{BcD}) = (1_{BcD'}\circ  BC\delta)\bullet\sigma_{Bcd}\,.
	\end{equation}
\end{itemize}
Finally, $B$ assigns (functorially) to each 2-cell $\gamma:c\rightarrow \bar c$ a modification $B\gamma :B c\Rightarrow B \bar c $, which in $\mathcal{E}$ means:
\begin{itemize}
\item 2-cells $B\gamma D:B cD\Rightarrow B \bar c D$,\newline
 satisfying the modification axiom,
\begin{equation}
\sigma_{B\bar cd}\bullet(1_{BC'd}\circ  B\gamma D) = (B\gamma D' \circ  1_{BCd})\bullet\sigma_{Bcd}
\end{equation}
and the functoriality condition
\begin{align}
B1_cD &= 1_{BcD}\\
B(\bar\gamma\bullet \gamma)D &= B\bar\gamma D \bullet B\gamma D\,. \label{eq:functOfC}
\end{align}
\end{itemize}
Being a lax functor, $B$ has to provide the unit and composition comparison modifications given by data:
\begin{itemize}
	\item{\hspace{-0.2cm}\tiny(\textbf{f2})} unit 2-cells $\eta_{B1_CD}:1_{B CD}\Rightarrow B 1_CD$
	\item{\hspace{-0.2cm}\tiny(\textbf{f2})} composition 2-cells $\mu_{Bcc'D}:(B c'D)\circ (B cD)\Rightarrow (B c'\circ cD)$\newline
which, in addition to the naturality condition
\begin{align}
\mu_{B\bar c\bar c'D}\bullet (B\gamma'D \circ  B\gamma D) =  B(\gamma'\circ  \gamma)D \bullet \mu_{Bcc'D}
\end{align}
and modification axiom,
\begin{align}
\sigma\bullet (1 \circ  \eta) &= \eta\circ  1\\
\sigma\bullet (1 \circ  \mu) &= (\mu\circ  1) \bullet (1\circ  \sigma)\bullet (\sigma\circ  1)\label{eq:swapcomp2}
\end{align}
satisfy the unit and associativity axioms (\ref{eq:unit1})-(\ref{eq:comp1}).
\end{itemize}

An arrow $ b :B\rightarrow B'$, being a lax transformation between lax functors $B$ and $B'$, assigns to each $C\in \mathcal{C}$ a lax transformation $ b  C:B C\rightarrow B' C$ and to each $c:C\rightarrow C'$ a modification $ \sigma_{bc}: B' c\circ b  C\Rightarrow b  C'\circ Bc$, which means the following data in $\mathcal{E}$:
\begin{itemize}
	\item 1-cells $ bCD:B CD\rightarrow B' CD \label{eq:bCD}$
	\item{\hspace{-0.2cm}\tiny(\textbf{t1})} 2-cells $ \sigma_{bCd}:B' Cd\circ b  CD\Rightarrow  b  CD'\circ BCd$
	\item{\hspace{-0.2cm}\tiny(\textbf{t2})} 2-cells $ \sigma_{bcD}:B' cD\circ b  CD\Rightarrow  b  C'D\circ BcD$,\newline
subject to naturality
\begin{align}
\sigma_{b\bar c D}\bullet(B'\gamma D \circ  1_{bCD}) &= (1_{bC'D}\circ  B\gamma D)\bullet\sigma_{bcD}\label{eq:swapnat1}\\
\sigma_{bC d}\bullet(B'C \delta \circ  1_{bCD}) &= (1_{bCD'}\circ  BC \delta)\bullet\sigma_{bCd} \label{eq:swapnat2}	
\end{align}
lax transformation
\begin{align}
\sigma\bullet (\eta \circ  1) &= 1\circ  \eta\label{eq:bsigmaeta}\\
\sigma\bullet (\mu \circ  1) &= (1 \circ  \mu) \bullet (\sigma\circ  1)\bullet (1\circ  \sigma)\label{eq:bsigmamu}
\end{align}
and a modification
\begin{equation}
(1\circ \sigma)\bullet(\sigma\circ  1)\bullet(1\circ \sigma)
=(\sigma\circ  1)\bullet(1\circ \sigma)\bullet(\sigma\circ  1)
\label{eq:swapyb}
\end{equation}
axioms.
\end{itemize}

A 2-cell $\beta: b \rightarrow \bar b$ in $\textrm{Lax}(\mathcal{C},\textrm{Lax}(\mathcal{D},\mathcal{E}))$, being a modification, assigns to each $C\in \mathcal{C}$ a modification $\beta C: b  C\Rightarrow \bar b C$, which in $\mathcal{E}$ means 
\begin{itemize}
	\item 2-cells $\beta CD: b  CD\Rightarrow \bar b CD$,
with modification axioms,
\begin{align}
\sigma_{\bar b c D}\bullet(1_{B'cD} \circ  \beta CD) &= (\beta C'D\circ  1_{BcD})\bullet\sigma_{bcD}\label{eq:modifc}\\
\sigma_{\bar b C d}\bullet(1_{B'Cd} \circ  \beta CD) &= (\beta CD'\circ  1_{BCd})\bullet\sigma_{bCd}\,.\label{eq:modifd}
\end{align}
\end{itemize}

\subsection{Symmetries}

Denote by $\mathrm{(Op)Lax}_\mathrm{(op)}(\mathcal{D},\mathcal{E})$ the 2-category of (op)lax functors (first op), (op)lax natural transformations (subscript op) and modifications.
\begin{prop}
There are isomorphisms:
\begin{align}
\mathrm{Lax}_\mathrm{op}(\mathcal{D},\mathcal{E})
&\cong
\mathrm{Lax}
(\mathcal{D}^\mathrm{op},\mathcal{E}^\mathrm{op})^\mathrm{op} \label{eq:dualOp}\\
\mathrm{OpLax}_\mathrm{op}(\mathcal{D},\mathcal{E})
&\cong
\mathrm{Lax}
(\mathcal{D}^\mathrm{co},\mathcal{E}^\mathrm{co})^\mathrm{co}\label{eq:dualCo}\\
\textrm{Lax}(\mathcal{C},\textrm{Lax}_\mathrm{op}(\mathcal{D},\mathcal{E}))
&\cong \textrm{Lax}_\mathrm{op}(\mathcal{D},\textrm{Lax}(\mathcal{C},\mathcal{E}))
\label{eq:ontStreetIdentity}\\
\textrm{Lax}(\mathcal{C},\textrm{OpLax}_\mathrm{op}(\mathcal{D},\mathcal{E}))&\cong \textrm{OpLax}_\mathrm{op}(\mathcal{D},\textrm{Lax}(\mathcal{C},\mathcal{E}))\,.
\label{eq:ontOplax}
\end{align}
\end{prop}
\begin{proof}
Data and axioms for the LHS of (\ref{eq:dualOp}) (resp. (\ref{eq:dualCo})) are obtained from the beginning of Section \ref{sec:unpacking} until the equation (\ref{eq:functOfC}), by ignoring the letter $B$ in all the names, and reversing the direction of 2-cells for data marked by \textbf{(t1)} (resp. \textbf{(f1)} or \textbf{(t1)}). On the other hand, the data and axioms for the RHS of (\ref{eq:dualOp}) (resp. (\ref{eq:dualCo})) have reversed sources and targets of arrows (resp. 2-cells), compared to the diagram (\ref{diag:notation}), but they also live in $\mathcal{E}^\mathrm{op}$ (resp. $\mathcal{E}^\mathrm{co}$), rather than $\mathcal{E}$; interpreted in $\mathcal{E}$, they have reversed  2-cells marked by (\textbf{t1}) (resp. (\textbf{f1}) or (\textbf{t1})). A possibly easier way to see this is to draw string diagrams in $\mathcal{E}^\mathrm{op}$ (resp. $\mathcal{E}^\mathrm{co}$), and then flip them horizontally (resp. vertically).

To prove (\ref{eq:ontStreetIdentity}), observe that the data and axioms in Section \ref{sec:unpacking}, with \textbf{(t1)} 2-cells reversed (LHS), and second and third letter in all labels formally swapped, corresponds to the same data and axioms when $C$ (resp. $c$, $\gamma$) is substituted for $D$ (resp. $d$, $\delta$), and vice versa, and then \textbf{(t2)} 2-cells are reversed (RHS).

Similarly, in (\ref{eq:ontOplax}) reversing \textbf{(f1)} and \textbf{(t1)} 2-cells, followed by swapping positions in labels, leads the same result as swapping variables and then reversing 2-cells marked by \textbf{(f2)} and \textbf{(t2)}.

Once the directions for data are fixed, all axioms are determined in a unique way, and there is no need to analyse them separately.
\end{proof}
\begin{cor}
There are isomorphism:
\begin{align}
\mathrm{OpLax}(\mathcal{D},\mathcal{E})
&\cong
\mathrm{Lax}
(\mathcal{D}^\mathrm{co\,op},\mathcal{E}^\mathrm{co\,op})^\mathrm{co\,op}\\
\textrm{OpLax}(\mathcal{C},\textrm{Lax}_\mathrm{op}(\mathcal{D},\mathcal{E}))&\cong \textrm{Lax}_\mathrm{op}(\mathcal{D},\textrm{OpLax}(\mathcal{C},\mathcal{E}))\,.
\end{align}
\end{cor}
\begin{cor}
There are isomorphism:
\begin{align}
[\mathcal{D},\mathcal{E}]_\mathrm{ont}
&\cong
[\mathcal{D}^\mathrm{op},\mathcal{E}^\mathrm{op}]_\mathrm{lnt}^\mathrm{op} \label{eq:dualOpPs}\\
[\mathcal{D},\mathcal{E}]_\mathrm{ont}
&\cong
[\mathcal{D}^\mathrm{co},\mathcal{E}^\mathrm{co}]_\mathrm{lnt}^\mathrm{co}\label{eq:dualCoPs}\\
[\mathcal{C},[\mathcal{D},\mathcal{E}]_\mathrm{ont}]_\mathrm{lnt}
&\cong 
[\mathcal{D},[\mathcal{C},\mathcal{E}]_\mathrm{lnt}]_\mathrm{ont}\,.
\label{eq:ontPsIdent}
\end{align}
\end{cor}

\subsection{Reviewing computads} \label{sec:computads}
The content of this part is taken from \cite{Street1976}. We describe the major ideas and leave out the details. 
\begin{defn}(\cite{Street1976}, with a technical modification\footnote{We take all paths between two objects to be the nodes of $\mathcal{G}(G,G')$; that is, $\mathcal{G}(G,G')_0=(\mathcal{F}|\mathcal{G}|)(G,G')$.})
A computad $\mathcal{G}$ consists of a graph $|\mathcal{G}|$ (providing a set of objects $|\mathcal{G}|_0$ and a set of generating arrows $|\mathcal{G}|_1$), and for each pair of objects $G,G'\in |\mathcal{G}|_0$ a graph $\mathcal{G}(G,G')$ with a set nodes\footnote{$\mathcal{F}|\mathcal{G}|$ is the free category on a graph $|\mathcal{G}|$}
$\mathcal{G}(G,G')_0=(\mathcal{F}|\mathcal{G}|)(G,G')$ and a set of edges denoted $\mathcal{G}(G,G')_1$ (providing generating 2-cells).
\end{defn}
A computad morphism assigns all the data, respecting sources and targets, forming a category $\mathrm{Cmp}$.

There is a free 2-category $\mathcal{F}\mathcal{G}$ on the computad $\mathcal{G}$ that has the same objects as $\mathcal{G}$. Arrows between $G$ and $G'$ are ``paths'' between $G$ and $G'$; that is, elements of $\mathcal{G}(G,G')_0$. To define 2-cells, it is not enough to take the free category on $\mathcal{G}(G,G')$ since it does not take whiskering into account. Instead, consider the set of whiskered generating 2-cells
\begin{align*}
\mathcal{G}^1(G,G')=\{(p,\alpha,p')|&p\in\mathcal{G}(G,X)_0,\\
	&\alpha\in\mathcal{G}(X,X')_1,\\
	&p'\in\mathcal{G}(X',G')_0\}\,.
\end{align*}
Finally, to impose the middle of four interchange, take the set of whiskered pairs
\begin{align*}
\mathcal{G}^2(G,G')=\{(p,\alpha,p',\alpha',p'')|
	&p\in\mathcal{G}(G,X)_0,\\
	&\alpha\in\mathcal{G}(X,X')_1,\\
	&p'\in\mathcal{G}(X',X'')_0\\
	&\alpha'\in\mathcal{G}(X'',X''')_1,\\
	&p''\in\mathcal{G}(X''',G')_0\}
\end{align*}
and form a coequalizer in $\mathrm{Cat}$ to obtain the hom $(\mathcal{F}\mathcal{G})(G,G')$ 
\begin{align}
\mathcal{F}\mathcal{G}^2(G,G')\rightrightarrows \mathcal{F}\mathcal{G}^1(G,G')\rightarrow
(\mathcal{F}\mathcal{G})(G,G')
\end{align}
where the two parallel arrows are the two obvious ways to compose whiskered $\alpha$ with whiskered $\alpha'$; see \cite{Street1976} for details and the rest of the construction.

Given a 2-category $\mathcal{E}$, the underlying computad $\mathcal{U}\mathcal{E}$ has the underlying graph obtained from the underlying category of $\mathcal{E}$; that is, $|\mathcal{U}\mathcal{E}|=\mathcal{U}|\mathcal{E}|$, and the hom graphs have edges $(\mathcal{U}\mathcal{E})(E,E')(p,p')=\mathcal{E}(E,E')(\circ p,\circ \bar p)$, where $\circ p$ denotes the arrow in $\mathcal{E}$ obtained by composing the path $p$ in $\mathcal{E}$. Assignments $\mathcal{F}$ and $\mathcal{U}$ extend to morphisms and form an adjunction, giving a bijection between arrows in $\mathrm{Cmp}$ and $\textrm{2-Cat}$
\begin{align}\label{key}
T:\mathcal{G}\rightarrow\mathcal{U}\mathcal{E}\,\,
\pgetsLR \,\,
\hat T:\mathcal{F}\mathcal{G}\rightarrow\mathcal{E}\,.
\end{align}

Intuitively, the 2-category $\mathcal{F}\mathcal{U}\mathcal{E}$ is the 2-category of pasting diagrams in $\mathcal{E}$, and the counit of the adjunction is the operation of actual pasting to obtain a (2-)cell in $\mathcal{E}$.

\subsection{The tensor product computad}

The goal is to construct a computad $\mathcal{G}$ which has data analogous to the one in Section \ref{sec:unpacking}, and then to impose further identification of 2-cells in $\mathcal{F}\mathcal{G}$, analogous to the axioms (\ref{eq:functid})-(\ref{eq:swapcomp2}).
Consider the computad $\mathcal{G}$ defined by the following data:
\begin{itemize}
	\item a set $|\mathcal{G}|_0=\mathrm{Ob}\mathcal{C}\times \mathrm{Ob}\mathcal{D}$ of nodes, whose elements are denoted $C\boxtimes D$
	\item the set $|\mathcal{G}|_1((C,D),(C',D'))$ of edges consists of arrows in $\mathcal{C}(C,C')$ if $D=D'$, denoted $c\boxtimes D$, and arrows in $\mathcal{D}(D,D')$ if $C=C'$, denoted $C\boxtimes d$, otherwise it is empty. The concatenation of $c\boxtimes D$ and $C'\boxtimes d$, as an arrow in the free category on $|\mathcal{G}|$, will be denoted by
	  $\{ C\boxtimes D \xrightarrow{c\boxtimes D}{C'\boxtimes D} \xrightarrow{C'\boxtimes d}{C'\boxtimes D'}\}$, and the empty path on $C\boxtimes D$ by $\{C\boxtimes D\}$. When the meaning is clear from the context we omit the tensor product character. A concise way of expressing the collection of edges is as a disjoint union
	  \begin{equation}
	  	|\mathcal{G}|_1((C,D),(C',D'))=\mathcal{C}(C,C')\times \delta_{DD'}+\delta_{CC'}\times \mathcal{D}(D,D'),
	  \end{equation}
	  with $\delta_{XY}$ being an empty set when $X\neq Y$ and singleton $\{X\}$ when $X=Y$.
	\item 2-cells
	\begin{itemize}
		\item for each object $C$ of $\mathcal{C}$ and 2-cell $\delta :d\Rightarrow \bar d$ in $\mathcal{D}$,
		\begin{equation}
		C\boxtimes \delta:\{  CD \xrightarrow{Cd} CD'\} \Rightarrow \{  CD \xrightarrow{C \bar d} CD'\}
		\end{equation}
		
		\item for each object $D$ of $\mathcal{D}$ and 2-cell $\gamma :c\Rightarrow \bar c$ in $\mathcal{C}$,
		\begin{equation}
		\gamma \boxtimes D:\{  CD \xrightarrow{cD} C'D\} \Rightarrow \{  CD \xrightarrow{\bar c D} C'D\}
		\end{equation}
		\item{\hspace{-0.2cm}\tiny(\textbf{f1})} for each $(C,D)\in |\mathcal{G}|_0$, the unit comparisons
		\begin{align}
		\mathbf{id}_{C1_D}&: \{CD\}\Rightarrow \{ CD \xrightarrow{C 1_D} CD\}
		\end{align}
		\item{\hspace{-0.2cm}\tiny(\textbf{f2})} for each $(C,D)\in |\mathcal{G}|_0$, the unit comparisons
		\begin{align}
		\mathbf{id}_{1_CD}&: \{CD\}\Rightarrow \{ CD \xrightarrow{1_C D} CD \}
		\end{align}
		\item{\hspace{-0.2cm}\tiny(\textbf{f1})} for each $C\in \mathcal{C}$ and composable pair $(d,d')$ in $\mathcal{D}$, a composition comparison
		\begin{align}
		\mathbf{comp}_{Cdd'}&: \{ CD\xrightarrow{C d} CD' \xrightarrow{C d'} CD'' \}
		\Rightarrow \{ CD\xrightarrow{C\boxtimes (d'\circ d)} CD'' \}
		\end{align}
		\item{\hspace{-0.2cm}\tiny(\textbf{f2})} for each $D\in \mathcal{D}$ and composable pair $(c,c')$ in $\mathcal{C}$, a composition comparison
		\begin{align}
		\mathbf{comp}_{cc'D}&: \{ CD\xrightarrow{c D} C'D \xrightarrow{c' D} C''D \}
		\Rightarrow \{ CD\xrightarrow{(c'\circ c) \boxtimes D} C''D \}
		\end{align}
		\item{\hspace{-0.2cm}\tiny(\textbf{t1})} for each pair of 1-cells $(c,d)$,
		\begin{equation}
		\mathbf{swap}_{cd}:\{ CD \xrightarrow{cD} C'D \xrightarrow{C'd} C'D' \}
			\Rightarrow 
		\{ CD \xrightarrow{Cd} CD' \xrightarrow{cD'} C'D' \}.
		\end{equation}
	\end{itemize}
\end{itemize}

The 2-category $\mathcal{C} \boxtimes_{cmp} \mathcal{D}$ is obtained from $\mathcal{F}\mathcal{G}$, the free 2-category on the computad $\mathcal{G}$, by imposing identifications:
\begin{itemize}
	\item preservation of identity 2-cells
		\begin{align}
		C\boxtimes 1_d &= 1_{C\boxtimes d}\label{eq:TensorId1} \\
		1_c\boxtimes D &= 1_{c\boxtimes D}\label{eq:TensorId2}
		\end{align}
	\item distributivity of $\boxtimes$ over vertical composition
	\begin{align}
	(C\boxtimes \delta')\bullet (C\boxtimes \delta) &= C\boxtimes (\delta'\bullet \delta) \label{eq:TensorDist1} \\
	(\gamma' \boxtimes D)\bullet (\gamma \boxtimes D) &= (\gamma'\bullet \gamma)\boxtimes D \label{eq:TensorDist2}
	\end{align}
	\item compatibility with the composition comparison 2-cells
	\begin{align}
	\mathbf{comp}_{C\bar d\bar d'}\bullet (C\boxtimes \delta' \circ  C\boxtimes \delta) &= C\boxtimes (\delta'\circ  \delta)\bullet\mathbf{comp}_{Cdd'}\label{eq:TensorComp1}\\
	\mathbf{comp}_{\bar c\bar c'D}\bullet (\gamma' \boxtimes D \circ  \gamma\boxtimes D) &=  (\gamma'\circ  \gamma)\boxtimes D \bullet \mathbf{comp}_{cc'D}\label{eq:TensorComp2}
	\end{align}
	\item compatibility with the swapping 2-cells 
	\begin{align}
	\mathbf{swap}_{\bar c\bar d}\bullet (C'\boxtimes \delta \circ  \gamma\boxtimes D)&=
	(\gamma \boxtimes D'\circ  C\boxtimes \delta)\bullet\mathbf{swap}_{cd} \label{eq:TensorSwap}
	\end{align}
	\item unit and associativity laws
	\begin{align}
		\mathbf{comp}\bullet (1 \circ  \mathbf{id}) = 1 &\,\,\mathrm{\&}  \,\, \mathbf{comp}\bullet (\mathbf{id} \circ  1)=1\label{eq:TensorUnit}\\
		\mathbf{comp}\bullet (\mathbf{comp} \circ  1)&=
		\mathbf{comp}\bullet (1 \circ  \mathbf{comp}) \label{eq:TensorAssoc}
	\end{align}
	\item compatibility of swapping with unit and composition
	\begin{align}
		\mathbf{swap}\bullet (1 \circ  \mathbf{id}) &= \mathbf{id}\circ  1\label{eq:TensorId1Swap}\\
		\mathbf{swap}\bullet (\mathbf{id} \circ  1) &= 1\circ  \mathbf{id}\label{eq:Tensor1IdSwap}\\
		\mathbf{swap}\bullet (1 \circ  \mathbf{comp})
		&= (\mathbf{comp}\circ  1) \bullet (1\circ \mathbf{swap})\bullet (\mathbf{swap}\circ  1)\label{eq:TensorComp1Swap}\\
		\mathbf{swap}\bullet (\mathbf{comp} \circ  1) &= (1 \circ  \mathbf{comp}) \bullet (\mathbf{swap}\circ  1)\bullet (1\circ  \mathbf{swap})\,.\label{eq:Tensor1CompSwap}
		\end{align}
\end{itemize}

\begin{prop}\label{prop:compIso}
Let $\mathcal{C}$, $\mathcal{D}$ and $\mathcal{E}$ be 2-categories, $\mathcal{C} \boxtimes_{cmp} \mathcal{D}$ the 2-category defined above, then there is an isomorphism
\begin{equation}\label{eq:compIso}
\mathrm{Lax}(\mathcal{C},\mathrm{Lax}(\mathcal{D},\mathcal{E}))\cong [\mathcal{C} \boxtimes_{\mathrm{cmp}} \mathcal{D},\mathcal{E}]_\mathrm{lnt}\,.
\end{equation}
\end{prop}
\begin{proof}
The data for $\mathcal{G}$ and identifications when forming $\mathcal{C} \boxtimes_{cmp} \mathcal{D}$ correspond exactly to data and laws (\ref{eq:functid})-(\ref{eq:swapcomp2}) for $B\in \textrm{Lax}(\mathcal{C},\textrm{Lax}(\mathcal{D},\mathcal{E}))$ in the Section \ref{sec:unpacking}. So, giving $B$ corresponds to giving a computad map $B_{\mathrm{cmp}}:\mathcal{G}\rightarrow\mathcal{U}\mathcal{E}$ such that the strict 2-functor $\hat{B}_{\mathrm{cmp}}:\mathcal{F}\mathcal{G}\rightarrow\mathcal{E}$ respects the identifications (\ref{eq:TensorId1})-(\ref{eq:Tensor1CompSwap}), which corresponds to giving a strict 2-functor $\hat{B}:\mathcal{C}\boxtimes_{\mathrm{cmp}} \mathcal{D}\rightarrow\mathcal{E}$.

Define $\mathcal{E}^\mathcal{D}:=[\mathcal{D},\mathcal{E}]_\mathrm{ont}$. From (\ref{eq:ontPsIdent}) we get the following isomorphism
\begin{align}
[\mathcal{F}\mathcal{G},\mathcal{E}]^\mathcal{J}_\mathrm{lnt}
&\cong 
[\mathcal{F}\mathcal{G},\mathcal{E}^\mathcal{J}]_\mathrm{lnt}\,.
\end{align}
In particular, we have a bijection on objects, so for a free arrow $\mathcal{J}=\mathbb{I}(:=0\rightarrow 1)$, (resp. free 2-cell 
$\mathcal{J}=\mathbb{D}(:=0\overset{\rightarrow}{\underset{\rightarrow}{\,\,\,\scriptstyle\Downarrow\,\,\,\,}}1)$), we get a bijection between arrows (resp. 2-cells) of $[\mathcal{F}\mathcal{G},\mathcal{E}]_\mathrm{lnt}$ and 2-functors $\mathcal{F}\mathcal{G}\rightarrow \mathcal{E}^\mathcal{I}$ (resp. $\mathcal{F}\mathcal{G}\rightarrow \mathcal{E}^\mathcal{D}$).

Consider a lax natural transformation between 2-functors respecting identifications (\ref{eq:TensorId1})-(\ref{eq:Tensor1CompSwap}) (as above)
\begin{equation}
\hat{b}_{\mathrm{cmp}}:\hat{B}_{\mathrm{cmp}}\Rightarrow\hat{B'}_{\mathrm{cmp}}:
\mathcal{F}\mathcal{G}\rightarrow\mathcal{E}\,.
\end{equation}
It corresponds to a 2-functor
\begin{equation}
\hat{b}_{\mathrm{cmp}}^\mathrm{curry}:
\mathcal{F}\mathcal{G}\rightarrow\mathcal{E}^\mathbb{I}
\end{equation}
which corresponds to a lax natural transformation $b:B\Rightarrow B'$ - the correspondence goes as follows
\begin{align}
\mathcal{G}&\xrightarrow{b_{\mathrm{cmp}}^\mathrm{curry}}\mathcal{U}\mathcal{E}^\mathbb{I}\\
C\boxtimes D &\mapsto bCD\\
c\boxtimes D,\,C\boxtimes d &\mapsto \sigma_{bcD},\,\sigma_{bCd}\\
\gamma\boxtimes D,\,C\boxtimes \delta &\mapsto (B\gamma D,B'\gamma D),\,(BC\delta,B'C\delta)
\label{eq:CmpNats}\\
\mathbf{id}_{-},\,\mathbf{comp}_{-},\,\mathbf{swap}_{-} &\mapsto (\eta_{B-},\eta_{B'-}),\, (\mu_{B-},\mu_{B'-}),\,(\sigma_{B-},\sigma_{B'-})\,.
\label{eq:CmpICS}
\end{align}
The RHS of (\ref{eq:CmpNats}) (resp. (\ref{eq:CmpICS})) being 2-cells of $\mathcal{E}^\mathbb{I}$ is equivalent to (\ref{eq:swapnat1}) and (\ref{eq:swapnat2}) (resp. (\ref{eq:bsigmaeta}), (\ref{eq:bsigmamu}) and (\ref{eq:swapyb})). 
The 2-functor $\hat{b}_{\mathrm{cmp}}^\mathrm{curry}$ respects identifications (\ref{eq:TensorId1})-(\ref{eq:Tensor1CompSwap}) because its source and target do, and so it also corresponds to a 2-functor
\begin{equation}
\hat{b}^\mathrm{curry}:
\mathcal{C}\boxtimes_\mathrm{cmp}\mathcal{D}\rightarrow\mathcal{E}^\mathbb{I}
\end{equation}
which is equivalently a lax natural transformation
\begin{equation}
\hat{b}:\hat{B}\Rightarrow\hat{B'}:
\mathcal{C}\boxtimes_\mathrm{cmp}\mathcal{D}\rightarrow\mathcal{E}\,.
\end{equation}

Similarly, a modification
\begin{equation} 
\hat{\beta}_{\mathrm{cmp}}:
\hat{b}_{\mathrm{cmp}}\rightarrow\hat{\bar b}_{\mathrm{cmp}}:
\hat{B}_{\mathrm{cmp}}\Rightarrow\hat{B'}_{\mathrm{cmp}}:
\mathcal{F}\mathcal{G}\rightarrow\mathcal{E}
\end{equation}
corresponds to a 2-functor
\begin{equation}
\hat{\beta}_{\mathrm{cmp}}^\mathrm{curry}:
\mathcal{F}\mathcal{G}\rightarrow\mathcal{E}^\mathbb{D}
\end{equation}
which corresponds to a modification $\beta:b\rightarrow \bar b$ via
\begin{align}
\mathcal{G}&\xrightarrow{\beta_{\mathrm{cmp}}^\mathrm{curry}}\mathcal{U}\mathcal{E}^\mathbb{D}\\
C\boxtimes D &\mapsto \beta CD\\
c\boxtimes D,\,C\boxtimes d &\mapsto (\sigma_{bcD},\sigma_{\bar{b}cD}),\,(\sigma_{bCd},\sigma_{\bar{b}Cd})
\label{eq:CmpModsMod}\\
\gamma\boxtimes D,\,C\boxtimes \delta &\mapsto (B\gamma D,B'\gamma D),\,(BC\delta,B'C\delta)
\label{eq:CmpNatsMod}\\
\mathbf{id}_{-},\,\mathbf{comp}_{-},\,\mathbf{swap}_{-} &\mapsto (\eta_{B-},\eta_{B'-}),\, (\mu_{B-},\mu_{B'-}),\,(\sigma_{B-},\sigma_{B'-})\,.
\label{eq:CmpICSMod}
\end{align}
The RHS of (\ref{eq:CmpModsMod}) being 1-cells of $\mathcal{E}^\mathbb{D}$ is equivalent to modification axioms (\ref{eq:modifc}) and (\ref{eq:modifd}). 
The RHS of (\ref{eq:CmpNatsMod}) and (\ref{eq:CmpICSMod}) being 2-cells of $\mathcal{E}^\mathbb{D}$, and $\hat{\beta}_{\mathrm{cmp}}^\mathrm{curry}$ respecting identifications (\ref{eq:TensorId1})-(\ref{eq:Tensor1CompSwap}), are just componentwise properties of $\hat{b}_{\mathrm{cmp}}^\mathrm{curry}$ and $\hat{\bar b}_{\mathrm{cmp}}^\mathrm{curry}$.
\end{proof} 

\subsection{Dual strictifications}

Notice that all the data and identifications for $\mathcal{G}(=:\mathcal{G}_\mathrm{lax}^{\mathcal{C}\mathcal{D}})$,
apart from those involving $\mathbf{swap}$,
are invariant (up to relabelling) with respect to exchanging $\mathcal{C}$ and $\mathcal{D}$. However, if we exchange $\mathcal{C}$ and $\mathcal{D}$ and consider oplax natural transformations at the same time, we arrive at an isomorphic computad $\mathcal{G}_{oplax}^{\mathcal{D}\mathcal{C}}\cong \mathcal{G}_\mathrm{lax}^{\mathcal{C}\mathcal{D}}$, the isomorphism consisting of exchanging the two positions in all the labels. All the identifications are isomorphic as well. This directly leads us to observe
\begin{cor}\label{prop:dualIso}
There is an isomorphism
\begin{equation}
\mathcal{C}\boxtimes \mathcal{D}\cong 
(\mathcal{D}^\mathrm{op}\boxtimes \mathcal{C}^\mathrm{op})^\mathrm{op}\,.
\end{equation}
\end{cor}
\begin{proof}
The computad $\mathcal{G}_{oplax}^{\mathcal{D}\mathcal{C}}$, with its identifications, generates a 2-category strictifying
$\mathrm{Lax}_\mathrm{op}(\mathcal{D},\mathrm{Lax}_\mathrm{op}(\mathcal{C},\mathcal{E}))$. On the other hand,
\begin{align}
\mathrm{Lax}_\mathrm{op}(\mathcal{D},\mathrm{Lax}_\mathrm{op}(\mathcal{C},\mathcal{E}))
&\overset{(\ref{eq:dualOp})}{\cong} 
\mathrm{Lax}(\mathcal{D}^\mathrm{op},
	\mathrm{Lax}(\mathcal{C}^\mathrm{op},\mathcal{E}^\mathrm{op}))^\mathrm{op}\\
&\overset{(\ref{eq:compIso})}{\cong} 
[\mathcal{D}^\mathrm{op}\boxtimes\mathcal{C}^\mathrm{op},
	\mathcal{E}^\mathrm{op}]^\mathrm{op}_\mathrm{lnt}\\
&\overset{(\ref{eq:dualOpPs})}{\cong} 
[(\mathcal{D}^\mathrm{op}\boxtimes\mathcal{C}^\mathrm{op})^\mathrm{op},
	\mathcal{E}]_\mathrm{ont}\,.
\end{align}
\end{proof}
\begin{cor}
Given 2-categories $\mathcal{C}$ and $\mathcal{D}$ there are isomorphisms
\begin{align}
\mathrm{Lax}_\mathrm{op}(\mathcal{C},\mathrm{Lax}_\mathrm{op}(\mathcal{D},\mathcal{E}))&\cong
 [\mathcal{D}\boxtimes \mathcal{C},\mathcal{E}]_\mathrm{ont}\\
\mathrm{OpLax}_\mathrm{op}(\mathcal{C},\mathrm{OpLax}_\mathrm{op}(\mathcal{D},\mathcal{E}))&\cong
[(\mathcal{C}^\mathrm{co}\boxtimes\mathcal{D}^\mathrm{co})^\mathrm{co},\mathcal{E}]_\mathrm{ont}\\
\mathrm{OpLax}(\mathcal{C},\mathrm{OpLax}(\mathcal{D},\mathcal{E}))&\cong
[(\mathcal{D}^\mathrm{co}\boxtimes\mathcal{C}^\mathrm{co})^\mathrm{co},\mathcal{E}]_\mathrm{lnt}\,.
\end{align}
\end{cor}
When $\mathcal{C}=\mathcal{D}=1$, we get free distributive laws between monads with opmorphisms (opfunctors in \cite{Street1972}), between comonads with opmorphisms and between comonads with morphisms, respectively.

Now we consider strictification for the case when one of the homs has oplax functors - $\textrm{Lax}(\mathcal{C},\textrm{OpLax}(\mathcal{D},\mathcal{E}))
$. Consider a computad $\mathcal{G}_{m}$, obtained from $\mathcal{G}$ by reversing 2-cells marked by (\textbf{f1}) and changing identifications accordingly. It generates a mixed tensor product $\mathcal{C}\boxtimes^\mathrm{m}_\mathrm{cmp}\mathcal{D}$, which analogously to Proposition \ref{prop:compIso} and Corollary \ref{prop:dualIso} satisfies Corollary \ref{prop:dualMixed}.
\begin{cor}\label{prop:dualMixed}
There are isomorphisms:
\begin{align}
\mathrm{Lax}(\mathcal{C},\mathrm{OpLax}(\mathcal{D},\mathcal{E}))&\cong
[\mathcal{C}\boxtimes^\mathrm{m}_\mathrm{cmp} \mathcal{D},\mathcal{E}]_\mathrm{lnt}\\
\mathcal{C}\boxtimes^\mathrm{m}_\mathrm{cmp} \mathcal{D}&\cong 
(\mathcal{D}^\mathrm{co}\boxtimes^\mathrm{m}_\mathrm{cmp} \mathcal{C}^\mathrm{co})^\mathrm{co}\,.
\end{align}
\end{cor}
The cases based on this one are:
\begin{align}
\mathrm{OpLax}_\mathrm{op}(\mathcal{C},\mathrm{Lax}_\mathrm{op}(\mathcal{D},\mathcal{E}))&\cong
[\mathcal{D}\boxtimes^\mathrm{m}_\mathrm{cmp} \mathcal{C},\mathcal{E}]_\mathrm{ont}\\
\mathrm{Lax}_\mathrm{op}(\mathcal{C},\mathrm{OpLax}_\mathrm{op}(\mathcal{D},\mathcal{E}))
&\cong [(\mathcal{C}^\mathrm{op}\boxtimes^\mathrm{m}_{cmp}\mathcal{D}^\mathrm{op})^\mathrm{op},\mathcal{E}]_\mathrm{ont}\\
\mathrm{OpLax}(\mathcal{C},\mathrm{Lax}(\mathcal{D},\mathcal{E}))&\cong
[(\mathcal{D}^\mathrm{op}\boxtimes^\mathrm{m}_{cmp}\mathcal{C}^\mathrm{op})^\mathrm{op},\mathcal{E}]_\mathrm{lnt}\,.
\end{align}

Finally, when the two homs have different choice for the direction of natural transformations, there is no strictification tensor product, mainly because we have to choose a type of natural transformation for the strict hom. For example, note that the objects $B\in\textrm{Lax}(\mathcal{C},\textrm{Lax}_\mathrm{op}(\mathcal{D},\mathcal{E}))$ correspond to the objects $B\in[\mathcal{D}\boxtimes\mathcal{C},\mathcal{E}]_\mathrm{(l)(o)nt}$ but crossings in the former allow\footnote{which is a shorter notation for $B'cD'\circ bCD'\circ BCd\Rightarrow B'C'd\circ bC'D\circ BcD$} $c\circ b\circ d\Rightarrow d\circ b\circ c$ while crossings of the latter allow $c\circ d\circ b\Rightarrow b\circ d\circ c$ for lax and $b\circ c\circ d\Rightarrow d\circ c\circ b$ for oplax natural transformations, suggesting that this case cannot be strictified. In a similar way, $\textrm{Lax}(\mathcal{C},\textrm{OpLax}_\mathrm{op}(\mathcal{D},\mathcal{E}))$ does not permit strictifications.

\section{Simplicial approach} \label{sec:direct}

\subsection{B\'enabou construction of the 2-category of paths} \label{sec:paths}

Let $\mathcal{C}$ be a 2-category and $\mathcal{C}^\dagger$ the 2-category of ``paths'' in $\mathcal{C}$, consisting of the same objects as $\mathcal{C}$, and
arrows between $C$ and $C'$ are strict 2-functors $p$ representing paths in $\mathcal{C}$ between $C$ and $C'$; that is,
\begin{equation}
{[}n]\xrightarrow{p}\mathcal{C},\;\;p(0)=C,\;\;p(n)=C',
\end{equation}
where $[n]$ is an object of $\Delta_{\bot\top}$, for details see Appendix \ref{sec:simpaths}.
Denote by\footnote{We reserve $p_i$, without brackets, to mean the length of the image as in (\ref{eq:pathDecomp}).}
 $(p)_i$ the $i^\mathrm{th}$ component in the path
\begin{equation}
(p)_i=p\left((i-1)\rightarrow i\right)\,.
\end{equation}

The identity is a path of zero length on $C$:
\begin{align}
[0]&\xrightarrow{0_C}\mathcal{C}\\
0&\mapsto C
\end{align}
and composition is given by ``concatenation'',
\begin{equation}
(n',p')\circ(n,p) = (n+n',p+p')
\end{equation}
where $(p+p')_i=(p)_i$ if $i\leq n$ and $(p+p')_i=(p')_{i-n}$ otherwise. This composition is strictly associative and unital.

Finally, 2-cells between $(n,p)$ and $(\bar n,\bar p)$, are pairs $(\xi,\alpha)$ where $\xi:[\bar n]\rightarrow [n]$ is a morphism in $\Delta_{\bot\top}$ and $\alpha$ is an identity on components, oplax-natural transformation, shortly icon, introduced in \cite{Lack2010a}:
\begin{equation}\label{eq:iconDef}
\alpha: p\circ\xi \Rightarrow \bar p,\,.
\end{equation}
with 2-cell components on identity arrows restricted to $\alpha_{1_i}=1_{1_{\bar p (i)}}$, which is true for general (op)lax transformations between normal lax functors.
So, $\alpha$ is determined by $\bar n$ components on non-identity arrows:
\begin{equation}
\alpha_i:=\alpha_{(i-1)\rightarrow i}:
(p\circ\xi)\left((i-1)\rightarrow i\right)
\Rightarrow (\bar p)_i\,.
\end{equation}
Note that if $\xi(i)=\xi(i-1)$ then the source of the corresponding component of $\alpha$ is the identity; that is, $\alpha_i:1_{p\xi (i)}\Rightarrow(\bar p)_i$.
The identity 2-cells is given by $1_{(n,p)}=(1_{[n]},1_p)$. The vertical composite of
$(\xi,\alpha)$ and $(\bar\xi,\bar\alpha)$ is obtained by pasting, as in the diagram \ref{diag:pathMorphComp}.
\begin{equation}\label{diag:pathMorphComp}
\begin{tikzpicture}[
baseline=(current  bounding  box.center)]
\def \strx {2.0};
\def \stry {0.6};
\node (n) at (-1 * \strx ,2 * \stry) {$[n]$};
\node (nb) at (-1 * \strx ,0 * \stry) {$[\bar n]$};
\node (nbb) at (-1 * \strx ,-2 * \stry) {$[\bar{\bar n}]$};
\node (c) at (-0 * \strx ,0) {$\mathcal{C}$};
\path[->,>=angle 90]
(n) edge node[above] {$p$} (c);
\path[->,>=angle 90] 
(nb) edge  node[below] {$\bar p$} (c);
\path[->,>=angle 90] 
(nbb) edge  node[below] {$\bar{\bar p}$} (c);
\path[<-,>=angle 90]
(n) edge node [left] {$\xi$} (nb);
\path[<-,>=angle 90]
(nb) edge node [left] {$\bar \xi$} (nbb);

\node at (-0.7 * \strx,0.7 * \stry) {$\alpha\Downarrow $};
\node at (-0.7 * \strx,-0.7 * \stry) {$\bar{\alpha}\Downarrow $};
\end{tikzpicture}
\end{equation}
The horizontal composition is concatenation, analogous to the one for path (1-cells),  $(\xi',\alpha')\circ(\xi,\alpha) = (\xi+\xi',\alpha+\alpha')$,
where $(\alpha+\alpha')_i=\alpha_i$ if $i\leq n$, and $(\alpha+\alpha')_i=\alpha'_{i-n}$ otherwise.

\subsection{Tensor product simplicially}

We proceed to describe our main result: a model $\mathcal{C} \boxtimes_{sim}\mathcal{D}$ for the strictification tensor product and then show that it is isomorphic to $\mathcal{C}\boxtimes_\mathrm{cmp}\mathcal{D}$.

Objects of $\mathcal{C}\boxtimes_{sim}\mathcal{D}$ are pairs $(C;D)$ with $C\in \mathcal{C}$ and $D \in\mathcal{D}$.

An arrow in $\mathcal{C}\boxtimes_{sim}\mathcal{D}$ is a sextuple $(n,p,r;m,q,s)$. It consists of a path in $\mathcal{C}$ of length $n$, a path in $\mathcal{D}$ (of length $m$) 
\begin{equation}
p:[n]\rightarrow \mathcal{C},~~q:[m]\rightarrow \mathcal{D}
\end{equation}
and a way to combine them into a string of length $n+m$; that is, a shuffle
\begin{align}
[n]\xleftarrow{r}  [n+m]  \xrightarrow{s} [m]
\end{align}
where $r$ and $s$ satisfy a compatibility condition (\ref{eq:rscomp}) saying that one increases if and only if the other one does not. 

The identity (empty path) on $(C;D)$ is defined by taking $m=n=0$, $r=s=1_{[0]}$, and $p$ and $q$ pick the objects $C$ and $D$. Composition is defined by path concatenation, formally expressed as tensor product of shuffles.

Below is an example of a 1-cell $\{c_1,d_1,c_2,c_3,d_2\}:(C_1,D_1)\rightarrow(C_4,D_3)$ in $\mathcal{C}\boxtimes_\mathrm{dir}\mathcal{D}$. Here, $n=3$, $m=2$, $r:[5]\rightarrow [3]$ and $s:[5]\rightarrow [2]$ give the coordinates of the corresponding node in the path, and $p:[3]\rightarrow \mathcal{C}$ and $q:[2]\rightarrow \mathcal{D}$ are the obvious functors producing the paths $\{c_i\}_{i=1}^3$ and $\{d_i\}_{i=1}^2$ in $\mathcal{C}$ and $\mathcal{D}$. 

\begin{equation}\label{diag:1cell}
\begin{tikzpicture}[baseline=(current  bounding  box.center)]
\def \str {1.0}
\node (C1) at (1 * \str,0 * \str) {$C_1$};
\node (C2) at (3 * \str,0 * \str) {$C_2$};
\node (C3) at (5 * \str,0 * \str) {$C_3$};
\node (C4) at (7 * \str,0 * \str) {$C_4$};

\path[->,font=\scriptsize,>=angle 90]
(C1) edge node[above] {$c_1$} (C2)
(C2) edge node[above] {$c_2$} (C3)
(C3) edge node[above] {$c_3$} (C4);

\node (D1) at (-0.5 * \str,-1 * \str) {$D_1$};
\node (D2) at (-0.5 * \str,-2 * \str) {$D_2$};
\node (D3) at (-0.5 * \str,-3 * \str) {$D_3$};

\path[->,font=\scriptsize,>=angle 90]
(D1) edge node[left] {$d_1$} (D2)
(D2) edge node[left] {$d_2$} (D3);

\node (C1D1) at (1 * \str,-1 * \str) {$C_1D_1$};
\node (C2D1) at (3 * \str,-1 * \str) {$C_2D_1$};
\node (C2D2) at (3 * \str,-2 * \str) {$C_2D_2$};
\node (C3D2) at (5 * \str,-2 * \str) {$C_3D_2$};
\node (C4D2) at (7 * \str,-2 * \str) {$C_4D_2$};
\node (C4D3) at (7 * \str,-3 * \str) {$C_4D_3$};

\path[->,font=\scriptsize,>=angle 90]
(C1D1) edge node[above] {$c_1$} (C2D1)
(C2D1) edge node[left] {$d_1$} (C2D2)
(C2D2) edge node[above] {$c_2$} (C3D2)
(C3D2) edge node[above] {$c_3$} (C4D2)
(C4D2) edge node[left] {$d_2$} (C4D3);
\end{tikzpicture}
\end{equation}

A 2-cell 
\begin{align}\label{eq:ten2cell}
(\xi,\alpha;\rho,\beta):(n,p,r;m,q,s)\rightarrow(\bar n,\bar p,\bar r;\bar m,\bar q,\bar s)
\end{align}
consists of:
\begin{itemize}
\item a shuffle morphism, that is functors $\xi: [\bar n]\rightarrow [n]$, $\rho: [\bar m]\rightarrow [m]$ preserving the first and the last element and satisfying, for all $\bar i\leq \bar n+\bar m$,
\begin{align} \label{eq:2cellneq}
\min r^{-1}(\xi \bar r\bar i) \leq 
 \max s^{-1}(\rho \bar s\bar i)
\end{align}
a condition ensuring that there are no swaps of arrows from $\mathcal{C}$ and $\mathcal{D}$ in the wrong direction. The condition (\ref{eq:2cellneq}) is an explicitly written condition for the existence of the natural transformation (\ref{eq:shufmornat}).

\item path 2-cells, that is, icons $\alpha: p\circ \xi \Rightarrow \bar p$ and $\beta: q\circ\rho \Rightarrow \bar q$, as defined in section \ref{sec:paths}

\end{itemize}

Below is an example of a 2-cell.

\begin{equation}
\begin{tikzpicture}[
baseline=(current  bounding  box.center)]
\def \str {1.0}
\node (C1D1) at (0 * \str,0 * \str) {$C_1D_1$};
\node (C2D1) at (2 * \str,0 * \str) {$C_2D_1$};
\node (C2D2) at (4 * \str,0 * \str) {$C_2D_2$};
\node (C3D2) at (6 * \str,0 * \str) {$C_3D_2$};
\node (C4D2) at (8 * \str,0 * \str) {$C_4D_2$};
\node (C4D3) at (10 * \str,0 * \str) {$C_4D_3$};

\path[->,font=\scriptsize,>=angle 90]
(C1D1) edge node(f1)[below] {$c_1$} (C2D1)
(C2D1) edge node(g1)[below] {$d_1$} (C2D2)
(C2D2) edge node(f2)[below] {$c_2$} (C3D2)
(C3D2) edge node(f3)[below] {$c_3$} (C4D2)
(C4D2) edge node(g2)[below] {$d_2$} (C4D3);

\node (C1D1) at (1 * \str,-2 * \str) {$C_1D_1$};
\node (C2D1) at (3 * \str,-2 * \str) {$C_2D_1$};
\node (C2D2) at (5 * \str,-2 * \str) {$C_2D_3$};
\node (C3D2) at (7 * \str,-2 * \str) {$C_2D_3$};
\node (C4D3) at (9 * \str,-2 * \str) {$C_4D_3$};

\path[->,font=\scriptsize,>=angle 90]
(C1D1) edge node(f1p)[above] {$\bar c_1$} (C2D1)
(C2D1) edge node(g1p)[above] {$\bar d_1$} (C2D2)
(C2D2) edge node(f2p)[above] {$\bar c_2$} (C3D2)
(C3D2) edge node(f3p)[above] {$\bar c_3$} (C4D3);

\path[->,font=\scriptsize,>=angle 90]
(f1) edge (f1p)
(f2) edge (f3p)
(f3) edge (f3p)
(g1) edge (g1p)
(g2) edge (g1p);
\end{tikzpicture}
\end{equation}
The above graph represents two 1-cells and data of $\xi$ and $\rho$, and what remains is to specify icon components $\alpha_1:c_1\Rightarrow \bar c_1,$ $\alpha_2:1_{C_2}\Rightarrow \bar c_2 ,$ $\alpha_3:c_3\circ c_2\Rightarrow \bar c_3$ in $\mathcal{C}$ and $\beta_1:d_2\circ d_1\Rightarrow \bar d_1$ in $\mathcal{D}$.

Vertical composition and whiskerings are defined componentwise as in $\mathbf {Shuff}$, $\mathcal{C}^\dagger$ and $\mathcal{D}^\dagger$.

\subsection{As a limit}

The category $\mathcal{C} \boxtimes_{sim}\mathcal{D}$ is a limit of the following diagram in $\textrm{2-Cat}$.
\[
\begin{array}{rcccccccl}
\mathcal{C}^\dagger & \rightarrow & \Sigma\Delta_{\bot\top} & \leftarrow & \mathrm{FDL} 
				&\rightarrow & \Sigma\Delta_{\bot\top} & \leftarrow & \mathcal{D}^\dagger \\
C & \mapsto & * & \mapsfrom & * &\mapsto * & \mapsfrom & D\\
(n,p) & \mapsto & [n] & \mapsfrom & (n,m,s,r) &\mapsto [m] & \mapsfrom & (m,q)\\
(\xi,\alpha) & \mapsto & \xi & \mapsfrom & (\xi,\rho,\gamma) &\mapsto \rho & \mapsfrom & (\rho,\beta)
\end{array}
\]

\subsection{Isomorphism between two constructions}

This part is about proving the following proposition.
\begin{thm}\label{thm:mainIso}
There is an isomorphism
\begin{equation}
\mathcal{C}\boxtimes_{sim} \mathcal{D} \cong \mathcal{C}\boxtimes_{cmp} \mathcal{D}\,.
\end{equation}
\end{thm}

We shall define a computad morphism $T:\mathcal{G}\rightarrow \mathcal{U}(\mathcal{C}\boxtimes_{sim} \mathcal{D})$, show that the induced strict 2-functor $\hat T:\mathcal{F}\mathcal{G}\rightarrow \mathcal{C}\boxtimes_{sim} \mathcal{D}$ respects the identifications (\ref{eq:TensorId1})-(\ref{eq:Tensor1CompSwap}), and that any other 2-functor
$\hat V:\mathcal{F}\mathcal{G}\rightarrow \mathcal{E}$  respecting them factors uniquely through $\hat T$. Then, from the universal property of $\mathcal{C}\boxtimes_{cmp} \mathcal{D}$ it will follow that $\mathcal{C}\boxtimes_{cmp} \mathcal{D} \cong \mathcal{C}\boxtimes_{sim} \mathcal{D}$.

The computad morphism  $T:\mathcal{G}\rightarrow \mathcal{U}(\mathcal{C}\boxtimes_{sim} \mathcal{D})$ is defined on nodes by
\begin{align}
T(C\boxtimes D) &= (C;D)\,,
\end{align}
on edges by
\begin{align}
T(C\boxtimes d) &= (0,\{C\},\sigma_0^1;
					1,\{D\xrightarrow{d} D'\},1_{[1]})\\
T(c\boxtimes D) &= (1,\{C\xrightarrow{c} C'\},1_{[1]};
					0,\{D\},\sigma_0^1)\,,
\end{align}
on 2-cells inherited from $\mathcal{C}$ and $\mathcal{D}$ by
\begin{align} \label{eq:elemStart}
T(C\boxtimes \delta) = 
	(1_{[0]},\{\};1_{[1]},\{\delta\})  &:  
		(0,\{C\},\sigma_0^1;
		 1,\{D\xrightarrow{d} D'\},1_{[1]})\nonumber\\
		&\Rightarrow  
		(0,\{C\},\sigma_0^1;
		 1,\{D\xrightarrow{\bar d} D'\},1_{[1]})\\
T(\gamma \boxtimes D) = (1_{[1]},\{\gamma\};1_{[0]},\{\})&:
		(1,\{C\xrightarrow{c} C'\},1_{[1]};
		 0,\{D\},\sigma_0^1)\nonumber\\
		&\Rightarrow
		(1,\{C\xrightarrow{\bar c} C'\},1_{[1]};
		 0,\{D\},\sigma_0^1)
\end{align}
and on the comparison and swapping 2-cells by
\begin{align}
T(\mathbf{id}_{1_C D}) = 
	(\sigma_0^1,\{1_{1_C}\};1_{[0]},\{\})&: 
		(0,\{C\},1_{[0]};0,\{D\},1_{[0]})\nonumber\\
		&\Rightarrow
		(1,\{C\xrightarrow{1_C}C\},1_{[1]};
		 0,\{D\},\sigma_0^1)\\
T(\mathbf{id}_{C 1_D}) = 
	(1_{[0]},\{\};\sigma_0^1,\{1_{1_D}\})&:
		(0,\{C\},1_{[0]};0,\{D\},1_{[0]})\nonumber\\
		&\Rightarrow
		(0,\{C\},\sigma_0^1;1,\{D\xrightarrow{1_D}D\},1_{[1]})
\end{align}
\begin{align}
T(\mathbf{comp}_{c,c',D}) =& 
	(\partial_1^2,\{1_{c'\circ c}\};1_{[0]},\{\}):\nonumber\\
		&(2,\{C\xrightarrow{c}C'\xrightarrow{c'}C''\},1_{[2]};
			0,\{D\},!_{[2]\rightarrow [0]})\nonumber\\
		&\Rightarrow
		(1,\{C\xrightarrow{c'\circ c}C''\},1_{[2]};
			0,\{D\},\sigma_0^1)\\
T(\mathbf{comp}_{C,d,d'}) =&
	(1_{[0]},\{\};\partial_1^2,\{1_{d'\circ d}\}):\nonumber\\
		&(0,\{C\},!_{[2]\rightarrow [0]};
			2,\{D\xrightarrow{d}D'\xrightarrow{d'}D''\},1_{[2]})\nonumber\\
		&\Rightarrow
		(0,\{C\},\sigma_0^1;1,\{D\xrightarrow{d'\circ d}D''\},1_{[1]})
\end{align}
\begin{align}
T(\mathbf{swap}_{c,d}) = &
	(1_{[1]},\alpha=\{1_c\}; 1_{[1]}, \beta=\{1_d\}):\nonumber\\ 
		&(1,\{C\xrightarrow{c}C'\},\sigma_1^2;
			1,\{D\xrightarrow{d} D'\},\sigma_0^2)\nonumber\\
		&\Rightarrow  
		(1,\{C\xrightarrow{c}C'\},\sigma_0^2;
			1,\{D\xrightarrow{d} D'\},\sigma_1^2)\,.\label{eq:elemEnd}
\end{align}
To check that the last 2-cell is the valid one, write equation (\ref{eq:shufmornat}) as
\begin{align}
	L\sigma_1^2\circ 1\circ\sigma_0^2 = \partial_2^2\circ\sigma_0^2
	\Rightarrow \partial_0^2\circ \sigma_1^2
			= R\sigma_0^2\circ 1\circ\sigma_1^2\,.
\end{align}
The cells on the RHS of (\ref{eq:elemStart})-(\ref{eq:elemEnd}) will be called \textit{elementary 2-cells}.

To see that the induced strict 2-functor respects identifications (\ref{eq:TensorId1})-(\ref{eq:Tensor1CompSwap}), note that $T(\mathbf{id})$,  $T(\mathbf{comp})$, and  $T(\mathbf{swap})$ have trivial icon components, while the definition of $T$ on other parts of the computad have trivial components in $\mathbf{Shuff}$, and that the composition of 2-cells in $\mathcal{C} \boxtimes_{sim} \mathcal{D}$ is done independently in each of the components.

Given  a computad map $V:\mathcal{G}\rightarrow \mathcal{U}\mathcal{E}$, such that $\hat V:\mathcal{F}\mathcal{G}\rightarrow \mathcal{E}$ respects the identifications (\ref{eq:TensorId1})-(\ref{eq:Tensor1CompSwap}), form the following assignments $W:\mathcal{C} \boxtimes_{sim} \mathcal{D}\rightarrow \mathcal{E}$ on objects
\begin{align}
	W(C;D) &= V(C\boxtimes D)
\end{align}
and on elementary arrows
\begin{align}
	W(0,\{C\},\sigma_0^1;1,\{D\xrightarrow{d} D'\},1_{[1]})&= 
		W(T(C\boxtimes d)) = V(C\boxtimes d)\\
	W(1,\{C\xrightarrow{c} C'\},1_{[1]};0,\{D\},\sigma_0^1)&=
		W(T(c\boxtimes D)) = V(c\boxtimes D)\,.
\end{align}
Since every shuffle can be written uniquely as a sum of shuffles of unit length, the above assignment determines assignment on all 1-cells; given $(n,p,r;m,q,s)$, assign to it the composite given by (\ref{eq:W1cellAssig}).
\begin{align}\label{eq:W1cellAssig}
	W(n,p,r;m,q,s) &= \circ_{i=n+m}^{1}
	\begin{cases}
	   V((p)_i\boxtimes qsi)\mathrm{, if }s_i=0\\
	   V(pri\boxtimes (q)_i)\mathrm{, if }r_i=0
	\end{cases}
\end{align}
When $n=m=0$ we get that $W$ preserves identities; that is,
\begin{align}
W(1_{(C;D)})=1_{W(C;D)}\,.
\end{align}
Also, $W$ preserves composition
\begin{align}
&W(n',p',r';m',q',s')
	\circ W(n,p,r;m,q,s)=\nonumber\\
&\circ_{i'=n'+m'}^{1}
	\begin{cases}
	   V((p')_{i'}\boxtimes q's'i')\mathrm{, if }s'_{i'}=0\\
	   V(p'r'i'\boxtimes (q')_{i'})\mathrm{, if }r'_{i'}=0
	\end{cases}
\circ_{i=n+m}^{1}
	\begin{cases}
	   V((p)_i\boxtimes qsi)\mathrm{, if }s_i=0\\
	   V(pri\boxtimes (q)_i)\mathrm{, if }r_i=0
	\end{cases}\\
&=\circ_{i=n'+m'+n+m}^{1}
	\begin{cases}
		V((p')_{i}\boxtimes q's'i)\mathrm{, if }s'_{i}=0\mathrm{, and }i>n+m\\
		V(p'r'i\boxtimes (q')_{i})\mathrm{, if }r'_i=0\mathrm{, and }i>n+m\\
		V((p)_i\boxtimes qsi)\mathrm{, if }s_i=0\mathrm{, and }i\leq n+m\\
		V(pri\boxtimes (q)_i)\mathrm{, if }r_i=0\mathrm{, and }i\leq n+m
	\end{cases}\\
&=W(n'+n,p'+p,r'+r;m'+m,q'+q,s'+s)\\
&=W((n',p',r';m',q',s')
	\circ (n,p,r;m,q,s))\,.
\end{align}
Hence, it is a functor on the underlying categories.

The requirement that $WT=V$ determines the assignment on identities
\begin{align}
T(1_g)=1_{Tg}
\end{align}
on elementary 2-cells $T\pi$
\begin{equation}
W(T\pi)=V(\pi)
\end{equation}
and similarly on whiskered elementary 2-cells
\begin{equation}
W(Tg''\circ T\pi \circ Tg ):=Vg''\circ V\pi \circ Vg=V(g''\circ \pi \circ g)
\end{equation}
where $T\pi$ is an elementary 2-cell and $Tg$ and $Tg'$ are 1-cells.

Given any 2-cell $(\xi,\alpha;\rho,\beta)$, as in (\ref{eq:ten2cell}), choose a decomposition into whiskered elementary 2-cells in the following order, starting from the target 1-cell,
\begin{itemize}
\item elementary $\beta$, $j=\bar m,..,1$
\begin{align}
J_j&=1\circ T(\bar p \bar r j\boxtimes \beta_j)\circ 1\\
	&=(1_{[\bar n]},\{1_{p_1},..,1_{p_{\bar n}}\};
	  1_{[\bar m]},\{1_{q_1},..,\beta_j,..,1_{q_{\bar n}}\})
\end{align}
\item elementary $\alpha$, $i=\bar n,..,1$
\begin{align}
I_i&=1\circ T(\alpha_i \boxtimes \bar q \bar s i)\circ 1\\
	&=(1_{[\bar n]},\{1_{p_1},..,\alpha_i,..,1_{p_{\bar n}}\};
		1_{[\bar m]},\{1_{q_1},..,1_{q_{\bar n}}\})
\end{align}
\item comparisons in $\mathcal{D}$, $j=\bar m,..,1$
\begin{itemize}
	\item if $\bar\rho_j=0$ then 
	\begin{equation}
	L_{j,1}=1\circ T(\mathbf{id})\circ 1=:L^{(\mathbf{id})}_j
	\end{equation}
	\item if $\bar\rho_j\geq 2$, $k=\bar{\rho_j}-1,...,1$
	\begin{equation}
		L_{j,k}=1\circ T(\mathbf{comp})\circ 1=:L^{(\mathbf{comp})}_{j,k}
	\end{equation}
	This order corresponds to left bracketing.
	\item if $\bar\rho_j=1$ then $L_{j,1}=1$, and can be ignored.
\end{itemize}
\item comparisons in $\mathcal{C}$, $i=\bar n,..,1$
\begin{itemize}
\item if $\bar\xi_i=0$ then 
\begin{equation}
	K_{i,1}=1\circ T(\mathbf{id})\circ 1=:K^{(\mathbf{id})}_{j}
\end{equation}
\item if $\bar\xi_j\geq 2$, $k=\bar{\xi_j}-1,...,1$
	\begin{equation}
		K_{i,k}=1\circ T(\mathbf{comp})\circ 1=:K^{(\mathbf{comp})}_{j,k}
	\end{equation}
	This order corresponds to left bracketing.
\item if $\bar\xi_j=1$ then $K_{j,1}=1$, and can be ignored.
\end{itemize}
\item crossings - the remaining 2-cell to decompose has trivial icon components as well as trivial $\xi$ and $\rho$. In the relation tables - which define the two shuffles - elementary crossings correspond to switching ones to zeros, or, going backwards, switching zeros to ones. Let $(x,y)$ be the coordinates of the corresponding crossings, order them by $x-y$ and then (if the $x-y$ value is the same) by $x+y$. Our backward decomposition starts with the last crossing in the table. Denote them by $S_i$. 
\end{itemize}
Now, define
\begin{align}
W(\xi,\alpha;\rho,\beta)=
\circ_{i}W(J_i)
\circ_{i}W(I_i)
\circ_{i,j}W(L_{i,j})
\circ_{i,j}W(K_{i,j})
\circ_{i}W(S_i)
\end{align}
Given a composable pair of 2-cells, the composite of their their images under $W$, $W(\bar\xi,\bar\alpha;\bar\rho,\bar\beta)\circ W(\xi,\alpha;\rho,\beta)$, is equal to
\begin{align}  
\circ_{i}W(\bar J_i)
\circ_{i}W(\bar I_i)
\circ_{i,j}W(\bar L_{i,j})
\circ_{i,j}W(\bar K_{i,j})
\circ_{i}W(\bar S_i)\nonumber\\
\circ_{i}W(J_i)
\circ_{i}W(I_i)
\circ_{i,j}W(L_{i,j})
\circ_{i,j}W(K_{i,j})
\circ_{i}W(S_i)\label{eq:nonCanDec}
\end{align} 
which need not be in the canonical form. The assignment on the composite 2-cell 
\begin{equation}
(\xi\circ\bar\xi,\bar\alpha\bullet(\alpha\circ\bar{\xi});\rho\circ\bar\rho,\bar{\beta}\bullet(\beta\circ\bar{\rho}))
\end{equation}
is in the canonical form, and the two are equal which we show by ``bubble-sorting'' the decomposition (\ref{eq:nonCanDec}). In each step one of two cases can happen:
\begin{itemize}
\item the output (target of the elementary part) of the first 2-cell does not overlap with the input (source of the elementary part) of the second 2-cell.  Then we can write the vertical composite of their images as
\begin{align}
&W(Tg_5\circ T\bar g_4\circ Tg_3\circ T\pi_2 \circ Tg_1)\nonumber\\
&\bullet
W(Tg_5\circ T \pi_1 \circ Tg_3\circ Tg_2 \circ Tg_1)\nonumber\\
&=V(g_5\circ \pi_1 \circ g_3 \circ \pi_2 \circ g_1)=\nonumber\\
&W(Tg_5\circ T\pi_1 \circ Tg_3\circ T\bar g_2 \circ Tg_1)\nonumber\\
&\bullet
W(Tg_5\circ Tg_4 \circ Tg_3\circ T\pi_2 \circ Tg_1)
\end{align} 
meaning that we can change the order of their composition after suitably changing the whiskering 1-cells.
\item the output of the first 2-cell overlaps with the input of the second 2-cell. Depending on which elementary 2-cells meet, do an operation according to the following table.
\begin{equation}
\arraycolsep=1.1pt\def\arraystretch{1.2}
\begin{array}{|c|c|c|c|c|c|c|c|}
\hline
1^{st}\backslash 2^{nd}&
\bar J&
\bar I&
\bar L^{(\mathbf{id})}&
\bar L^{(\mathbf{comp})}&
\bar K^{(\mathbf{id})}&
\bar K^{(\mathbf{comp})}&
\bar S\\
\hline
J&
(\ref{eq:TensorDist1})&
\bot&
\bot&
(\ref{eq:TensorComp1})&
\bot&
\bot&
\bot/(\ref{eq:TensorSwap})\\
\hline
I&
\bot&
(\ref{eq:TensorDist2})&
\bot&
\bot&
\bot&
(\ref{eq:TensorComp2})&
(\ref{eq:TensorSwap})/\bot\\
\hline
L^{(\mathbf{id})}&
R&
\bot&
\bot&
(\ref{eq:TensorUnit})&
\bot&
\bot&
\bot/(\ref{eq:Tensor1IdSwap})\\
\hline
L^{(\mathbf{comp})}&
R&
\bot&
\bot&
R/(\ref{eq:TensorAssoc})&
\bot&
\bot&
\bot/(\ref{eq:Tensor1CompSwap})\\
\hline
K^{(\mathbf{id})}&
\bot&
R&
\bot&
\bot&
\bot&
(\ref{eq:TensorUnit})&
(\ref{eq:TensorId1Swap})/\bot\\
\hline
K^{(\mathbf{comp})}&
\bot&
R&
\bot&
\bot&
\bot&
R/(\ref{eq:TensorAssoc})&
(\ref{eq:TensorComp1Swap})/\bot\\
\hline
S&
\bot / R&
R / \bot&
\bot&
\bot / \bot / R&
\bot&
R/\bot/\bot&
R/\bot/\bot\\
\hline
\end{array} 
\end{equation}
If the first 2-cell has $n$ outputs and the second 2-cell has $m$ inputs, there are $n+m-1$ ways to match them. When different, these cases are separated by ``$/$''. The symbol $\bot$ denotes that matching is not possible for that case, and $R$ denotes that the matching is possible, but the order is already correct (lower triangle). Finally, an equation number tells us to apply apply $\hat T$ to both sides, and substitute the LHS, which appears in the composition, with the RHS. Each step changes the decomposition of the 2-cell, and the fact that $\hat{V}$ preserves relations ensures that the composite in $\mathcal{E}$ does not change.
\end{itemize}
This proves that $W$ is functorial on homs.

A 2-cell in $\mathcal{C}\boxtimes \mathcal{D}$, obtained by whiskering, has the same elementary 2-cells in its decomposition as the original 2-cell. Hence, the two different composites 
\begin{equation}
(WT\bar g'\circ W(\xi,\alpha;\rho,\beta)) \bullet
(W(\xi',\alpha';\rho',\beta')\circ WTg)
\end{equation}
and
\begin{equation}
(W(\xi',\alpha';\rho',\beta')\circ WT\bar g) \bullet
(WTg'\circ W(\xi,\alpha;\rho,\beta))
\end{equation}
necessarily bubble-sort to $W((\xi',\alpha';\rho',\beta')\circ (\xi,\alpha;\rho,\beta))$. This completes the proof that $W$ is a 2-functor.

The functor $\hat T$ is bijective on objects and arrows, and surjective on 2-cells, so $W$ is the unique 2-functor satisfying $\hat V=W\hat T.$

\subsection{Mixed tensor product}

The case covering the free mixed distributive law, strictifying 
$\textrm{Lax}(\mathcal{C},\textrm{OpLax}(\mathcal{D},\mathcal{E}))$, produces $\mathcal{C}\boxtimes^\mathrm{m}_\mathrm{dir}\mathcal{D}$ that has the same objects and arrows as $\mathcal{C}\boxtimes\mathcal{D}$, and 2-cells differ by changing the direction of  $\rho:[m]\rightarrow [\bar m]$ to accommodate comultiplication and counit, change in icon $\beta:q\Rightarrow\bar q \rho:[m]\rightarrow\mathcal{D}$, with the restriction for crossings taking a slightly different form
\begin{equation}
Lr \circ \xi \circ \bar r\Rightarrow Rs \circ R\rho \circ \bar s\,.
\end{equation}
With a proof following the same steps as the non-mixed case, we state the following proposition.
\begin{prop}
There is an isomorphism
\begin{equation}
\mathcal{C}\boxtimes_\mathrm{dir}^\mathrm{m} \mathcal{D} \cong \mathcal{C}\boxtimes_\mathrm{cmp}^\mathrm{m} \mathcal{D}\,.
\end{equation}
\end{prop}

\section{Some properties and an example}

\subsection{Lax monoidal structure}

In this section we will recall the universal property of lax Gray tensor product \cite{Gray1974}, and use it together with B\'enabou construction of paths from Section \ref{sec:paths} to describe lax monoidal structure on the category of 2-categories and lax functors.

Let $\textrm{L2-Cat}$ denote the category of (small) 2-categories and lax functors, while $\textrm{2-Cat}$ denotes denote the subcategory of $\textrm{L2-Cat}$ consisting of strict 2-functors. The inclusion $i_0:\textrm{2-Cat}\hookrightarrow \textrm{L2-Cat}$ has a left adjoint:
\begin{itemize}
\item There is an assignment on objects $(-)^\dagger:\textrm{L2-Cat}\hookrightarrow \textrm{2-Cat}$ 
(Benabou strictification construction, Section \ref{sec:paths})
\item For each $\mathcal{C}$ there is an universal L2-Cat arrow (lax functor) $\eta_\mathcal{C}:\mathcal{C}\rightarrow \mathcal{C}^\dagger$, meaning, each lax functor $F:\mathcal{C}\rightarrow \mathcal{D}$ gives rise to a unique strict functor
\begin{align}\label{eq:strict}
s_0F:\mathcal{C}^\dagger\rightarrow \mathcal{D}
\end{align}
satisfying
$
s_0F\circ\eta_\mathcal{C}=F\,.
$ 
\end{itemize}
In fact, one could define computad presentation of $(-)^\dagger$ and analogously to proofs of Proposition \ref{prop:compIso} and Theorem \ref{thm:mainIso} show that
\begin{equation}\label{eq:ben}
\mathrm{Lax}(\mathcal{C},\mathcal{E})\cong [\mathcal{C}^\dagger,\mathcal{E}]_\mathrm{lnt}\,.
\end{equation}

Lax Gray tensor product, $\otimes_l :\textrm{2-Cat}\times \textrm{2-Cat}\rightarrow \textrm{2-Cat}$, is a tensor product for the internal hom $[-,-]_\mathrm{lnt}$, that is
\begin{equation}\label{eq:gray}
[\mathcal{C},[\mathcal{D},\mathcal{E}]_\mathrm{lnt}]_\mathrm{lnt}
\cong [\mathcal{C}\otimes_l\mathcal{D},\mathcal{E}]_\mathrm{lnt}\,.
\end{equation}

The left hand side of Eq.~(\ref{eq:mainiso}) can be transformed
\begin{align}
\mathrm{Lax}(\mathcal{C},\mathrm{Lax}(\mathcal{D},\mathcal{E}))
	&\stackrel{\mathrm{(\ref{eq:ben})}}{\cong} [\mathcal{C}^\dagger,[\mathcal{D}^\dagger,\mathcal{E}]_\mathrm{lnt}]_\mathrm{lnt}\\
	&\stackrel{\mathrm{(\ref{eq:gray})}}{\cong} [\mathcal{C}^\dagger\otimes_l\mathcal{D}^\dagger,\mathcal{E}]_\mathrm{lnt}
\end{align}
leading to the third description of the tensor product\footnote{The lax Gray product $\otimes_l$ is defined via its universal property, and the explicit description involves relations and quotienting. Our direct description, explained in Section \ref{sec:direct}, involves no quotienting.}
\begin{equation}
	\mathcal{C}\boxtimes \mathcal{D} \cong \mathcal{C}^\dagger\otimes_l\mathcal{D}^\dagger\,.
\end{equation}

From 2-monadic point of view, the $\otimes_l$ is a pseudo algebra on $\textrm{2-Cat}$ for the monoidal category 2-monad on $\textrm{CAT}$. The adjunction $(-)^\dagger \dashv i_0$ induces a lax algebra structure on $\textrm{L2-Cat}$ given by
\begin{equation}
\boxtimes_n (\mathcal{C}_1,\ldots,\mathcal{C}_n):= \mathcal{C}_1^\dagger \otimes_\text{l}\ldots\otimes_\text{l}\mathcal{C}_n^\dagger\,.
\end{equation}

\subsection{Generalization of the composite monad}

There is an obvious 2-functor $L:\mathcal{C}\boxtimes \mathcal{D}\rightarrow\mathcal{C}\times\mathcal{D}$ that forgets shuffles and composes paths. It has a right adjoint $R$ in the 2-category of 2-categories, lax functors and icons:
\begin{align}
\mathcal{C} \times \mathcal{D} & \xrightarrow{R}  \mathcal{C} \boxtimes \mathcal{D}\\
(C,D) &\mapsto C\boxtimes D\\
(c,d) &\mapsto CD\xrightarrow{Cd}CD'\xrightarrow{cD'}C'D'\\
(\gamma,\delta) &\mapsto (\gamma\boxtimes D') \circ (C\boxtimes \delta)
\end{align}
with identity and composition comparison maps
\begin{align}
!&:1_{C \boxtimes D}  \Rightarrow CD\xrightarrow{C 1_D}CD\xrightarrow{1_CD}CD\\
(\partial^2_1,{1};\partial^2_1,{1})&: CD\xrightarrow{Cd}CD'\xrightarrow{cD'}C'D'\xrightarrow{C'd'}C'D''
\xrightarrow{c'D''}C''D''\\
& \Rightarrow  CD\xrightarrow{C(d'\circ d)}CD''\xrightarrow{(c'\circ c)D''}C''D''\,.
\end{align}
The composite $L\circ R$ is just the identity functor $1_{\mathcal{C}\times \mathcal{D}}$, while the unit of the adjunction is an icon
\begin{equation}\label{eq:unitRL}
\eta:1_{\mathcal{C}\boxtimes \mathcal{D}}\Rightarrow R\circ L
\end{equation}
assigning to each arrow $(n,p,r;m,q,s)$ in $\mathcal{C}\boxtimes \mathcal{D}$ a 2-cell
\begin{align}
(!_{[1]\rightarrow[n]},1_{\circ p};!_{[1]\rightarrow[m]},1_{\circ q}):(n,p,r;m,q,s)\Rightarrow (1,\circ p,\sigma^2_0;1,\circ q,\sigma^2_1)\,.
\end{align}
Whiskering $\eta$ on the left (resp. right) by $L$ (resp. $R$) gives the identity on $L$ (resp. $R$), proving the adjunction axioms.
 
Any strict functor $\hat B:\mathcal{C}\boxtimes \mathcal{D}\rightarrow \mathcal{E}$ can be precomposed with $R$ to give a lax functor
\begin{align}
\hat B\circ R:\mathcal{C}\times \mathcal{D}\rightarrow \mathcal{E}\,.
\end{align}
This generalizes the notion of a composite monad induced by a distributive law.

\subsection{Parametrizing parametrization of categories}

 Take $\mathcal{C}$ and $\mathcal{D}$ to be just categories (seen as locally discrete 2-categories), and\footnote{Instead of $\mathrm{Span}$ one can take a strict version with objects sets $X,Y...$ and arrows cocontinuous functors $\mathrm{Set}/X\rightarrow \mathrm{Set}/Y$ which are determined by the assignment of singletons.}
 $\mathcal{E}=\mathrm{Span}$.

The bicategory of spans is equivalent to the bicategory of matrices, which is in turn a full subcategory of\footnote{Consisting of categories and modules (aka profunctors or distributors)}
$\mathrm{Mod}$. Each strict functor $\hat B:\mathcal{C}\boxtimes \mathcal{D}\rightarrow \mathrm{Span}$ is, in particular, a normal lax functor, so we can use the B\'enabou construction \cite{Street2001} (after forgetting 2-cells) to obtain a category $\tilde{B}_\mathrm{nerve}$ parametrised over $\mathcal{C}\boxtimes \mathcal{D}$. Explicitly, $\tilde{B}_\mathrm{nerve}$ has objects over $C\boxtimes D$ given by the set $BCD$. Arrows over $C\boxtimes d$ and $c\boxtimes D$ are elements of spans $BCd$ and $BcD$ respectively, and they generate arrows over arbitrary paths, which are, due to composition in $\mathrm{Span}$, composable tuples.

The 2-cells that we have temporarily forgotten are mapped to span morphisms. In particular, the images $\hat B\eta_p$ of the unit of the adjunction (\ref{eq:unitRL}) give a unique way of ``composing'' arbitrary arrows in $\tilde{B}_\mathrm{nerve}$, resulting in an arrow over a path in $\mathcal{C}\boxtimes \mathcal{D}$ of the form $CD\xrightarrow{Cd}CD'\xrightarrow{cD}C'D'$. The image of this assignment forms a category $\tilde{B}$ whose composition is concatenation in $\tilde{B}_\mathrm{nerve}$ followed by applying (the unique) appropriate $B\eta$. Uniqueness guarantees the identity and associativity laws.

Explicitly, $\tilde B$ with the same objects as $\tilde{B}_\mathrm{nerve}$, and arrows between $X\in B(C\boxtimes D)$ and $X'\in B(C'\boxtimes D')$ are elements of $B(CD\xrightarrow{Cd}CD'\xrightarrow{cD}C'D')$, denoted by pairs $(g,f)$,. The identity is
\begin{align}
1_X&=(1_X^D,1_X^C),\mathrm{ with}\\
1_X^D&:=(B\mathbf{id}_{C1_D})(X)\\
1_X^C&:= (B\mathbf{id}_{1_CD})(X)
\end{align}
and composition is given by
\begin{align}
(g',f')\circ (g,f)=B((\mathbf{comp}\circ\mathbf{comp})\bullet (1\circ \mathbf{swap}\circ 1))(g',f',g,f)\,.
\end{align}

For each object $D\in \mathcal{D}$ we get a subcategory $\pi_D\tilde{B}$ parametrized by $\mathcal{C}$ - an object $X$ over $C$ is an element of $BCD$, and arrow $f:X\rightarrow X'$ over $c$ is an element of $BcD$, which can be identified with an arrow $(1_X^D,f)$ of $\tilde B$. Similarly, each object $C\in \mathcal{C}$ gives a subcategory $\pi_C\tilde{B}$, parametrized by $\mathcal{D}$.
Furthermore, each arrow $(g,f)$ in $\tilde{B}$ can be decomposed as
\begin{equation}
(1_{D'},f)\circ(g,1_C)
\end{equation}
or as
\begin{equation}
(g,1_{C'})\circ(1_{D},f)\,.
\end{equation}

\appendix

\section{Simplices, intervals and shuffles} \label{sec:simpaths}

The {\it algebraist's delta}, denoted by $\Delta_{a}$, is the full subcategory of $\mathbf{Cat}$ consisting of categories
$\langle n \rangle$ whose objects are numbers $0,...,n-1$ and 1-cells are unique $i\rightarrow j$ when $i\leq j$. The empty category is denoted $\langle 0 \rangle$. Arrows between $\langle n \rangle$ and $\langle n' \rangle$ are functors; that is, order preserving functions, generated by face and degeneracy maps
\begin{align}
\sigma_i^n&:\langle n+1 \rangle  \rightarrow\langle n \rangle ,\, i=0,\ldots,n-1\\
\partial_i^n&:\langle n \rangle \rightarrow\langle n+1 \rangle ,\, i=0,\ldots,n
\end{align}
which can be presented in a diagram
\begin{equation}\label{diag:Deltaa}
\begin{tikzpicture}[
baseline=(current  bounding  box.center)]
\def \strx {2};
\def \stry {0.5};
\def \spcnd {0.25};
\def \ellip {(0.2 and 0.2)}
\node (G0) at (0 * \strx ,0) {$\langle 0 \rangle $};
\node (G1) at (1 * \strx ,0) {$\langle 1 \rangle $};
\node (G2) at (2 * \strx ,0) {$\langle 2 \rangle $};
\node (G3) at (3 * \strx ,0) {$\langle 3 \rangle $};
\node (etc) at (3.5 * \strx ,0) {$\ldots$};

\path[->,font=\scriptsize,>=angle 90]
(0 * \strx + \spcnd * \strx ,0) edge (1 * \strx - \spcnd * \strx,0);
\draw[fill=white] (0.5 * \strx, 0 * \stry) ellipse \ellip [fill=white] node {\scriptsize $\partial_0^0$};

\path[->,font=\scriptsize,>=angle 90]
(1 * \strx + \spcnd * \strx , 1 * \stry) edge 
(2 * \strx - \spcnd * \strx, 1 * \stry);
\draw[fill=white] (1.5 * \strx, 1 * \stry) ellipse \ellip [fill=white] node {\scriptsize $\partial_1^1$};
\path[<-,font=\scriptsize,>=angle 90]
(1 * \strx + \spcnd * \strx ,0) edge 
(2 * \strx - \spcnd * \strx,0);
\draw[fill=white] (1.5 * \strx, 0 * \stry) ellipse \ellip [fill=white] node {\scriptsize $\sigma_0^1$};
\path[->,font=\scriptsize,>=angle 90]
(1 * \strx + \spcnd * \strx ,-1 * \stry) edge 
(2 * \strx - \spcnd * \strx,-1 * \stry);
\draw[fill=white] (1.5 * \strx, -1 * \stry) ellipse \ellip [fill=white] node {\scriptsize $\partial_0^1$};


\path[->,font=\scriptsize,>=angle 90]
(2 * \strx + \spcnd * \strx ,2 * \stry) edge
(3 * \strx - \spcnd * \strx, 2 * \stry);
\draw[fill=white] (2.5 * \strx, 2 * \stry) ellipse \ellip [fill=white] node {\scriptsize $\partial_2^2$};
\path[<-,font=\scriptsize,>=angle 90]
(2 * \strx + \spcnd * \strx , 1 * \stry) edge
(3 * \strx - \spcnd * \strx, 1 * \stry);
\draw[fill=white] (2.5 * \strx, 1 * \stry) ellipse \ellip [fill=white] node {\scriptsize $\sigma_1^2$};
\path[->,font=\scriptsize,>=angle 90]
(2 * \strx + \spcnd * \strx ,0 * \stry) edge
(3 * \strx - \spcnd * \strx, 0 * \stry);
\draw[fill=white] (2.5 * \strx, 0 * \stry) ellipse \ellip [fill=white] node {\scriptsize $\partial_1^2$};
\path[<-,font=\scriptsize,>=angle 90]
(2 * \strx + \spcnd * \strx ,-1 * \stry) edge
(3 * \strx - \spcnd * \strx,-1 * \stry);
\draw[fill=white] (2.5 * \strx, -1 * \stry) ellipse \ellip [fill=white] node {\scriptsize $\sigma_0^2$};
\path[->,font=\scriptsize,>=angle 90]
(2 * \strx + \spcnd * \strx ,-2 * \stry) edge
(3 * \strx - \spcnd * \strx,-2 * \stry);
\draw[fill=white] (2.5 * \strx, -2 * \stry) ellipse \ellip [fill=white] node {\scriptsize $\partial_0^2$};

\end{tikzpicture}
\end{equation}
A natural transformation between $f$ and $\bar f$, if one exists, is unique and witnesses that $fi\leq \bar fi$ for all $i$, turning  $\Delta_a [\langle n \rangle ,  \langle n' \rangle ]$ into a poset. The 2-category $\Delta_{a}$ is equipped with a strict monoidal structure, the ordinal sum $\oplus$.

\subsection{Intervals - free monoid}
\label{sec:FM}

Denote by $\Delta_{\bot\top}$ the subcategory of $\Delta_{a}$, called the category of intervals, consisting of relabelled objects 
\begin{equation}
[n]:= \langle n+1 \rangle ,\,n=0,1,... 
\end{equation}
and 1-cells that preserve the first and the last element; it is generated by the arrows from the inside of the diagram (\ref{diag:Deltaa}), represented by the bold part of
\begin{equation}
\begin{tikzpicture}[
baseline=(current  bounding  box.center)]
\def \strx {2};
\def \stry {0.5};
\def \spcnd {0.25};
\def \ellip {(0.2 and 0.2)}
\node (G0) at (0 * \strx ,0) {$\cdot$};
\node (G1) at (1 * \strx ,0) {$\mathbf{[0]}$};
\node (G2) at (2 * \strx ,0) {$\mathbf{[1]}$};
\node (G3) at (3 * \strx ,0) {$\mathbf{[2]}$};
\node (etc) at (3.5 * \strx ,0) {$\ldots$};

\path[->,font=\scriptsize,>=angle 90,line width=0.1pt]
(0 * \strx + \spcnd * \strx ,0) edge (1 * \strx - \spcnd * \strx,0);

\path[->,font=\scriptsize,>=angle 90,line width=0.1pt]
(1 * \strx + \spcnd * \strx , 1 * \stry) edge 
(2 * \strx - \spcnd * \strx, 1 * \stry);
\path[<-,font=\scriptsize,>=angle 90,line width=1pt]
(1 * \strx + \spcnd * \strx ,0) edge 
(2 * \strx - \spcnd * \strx,0);
\draw[fill=white] (1.5 * \strx, 0 * \stry) ellipse \ellip [fill=white] node {\scriptsize $\sigma_0^1$};
\path[->,font=\scriptsize,>=angle 90,line width=0.1pt]
(1 * \strx + \spcnd * \strx ,-1 * \stry) edge 
(2 * \strx - \spcnd * \strx,-1 * \stry);


\path[->,font=\scriptsize,>=angle 90,line width=0.1pt]
(2 * \strx + \spcnd * \strx ,2 * \stry) edge
(3 * \strx - \spcnd * \strx, 2 * \stry);
\path[<-,font=\scriptsize,>=angle 90,line width=1pt]
(2 * \strx + \spcnd * \strx , 1 * \stry) edge
(3 * \strx - \spcnd * \strx, 1 * \stry);
\draw[fill=white] (2.5 * \strx, 1 * \stry) ellipse \ellip [fill=white] node {\scriptsize $\sigma_1^2$};
\path[->,font=\scriptsize,>=angle 90,line width=1pt]
(2 * \strx + \spcnd * \strx ,0 * \stry) edge
(3 * \strx - \spcnd * \strx, 0 * \stry);
\draw[fill=white] (2.5 * \strx, 0 * \stry) ellipse \ellip [fill=white] node {\scriptsize $\partial_1^2$};
\path[<-,font=\scriptsize,>=angle 90,line width=1pt]
(2 * \strx + \spcnd * \strx ,-1 * \stry) edge
(3 * \strx - \spcnd * \strx,-1 * \stry);
\draw[fill=white] (2.5 * \strx, -1 * \stry) ellipse \ellip [fill=white] node {\scriptsize $\sigma_0^2$};
\path[->,font=\scriptsize,>=angle 90,line width=0.1pt]
(2 * \strx + \spcnd * \strx ,-2 * \stry) edge
(3 * \strx - \spcnd * \strx,-2 * \stry);
\end{tikzpicture}
\end{equation}

It is clear that suspension (moving nodes to the left) gives an isomorphism
\begin{align}
\Delta_{\bot\top}^{op} &\cong \Delta_{a}\label{eq:isoDelta}\\
{[}n{]}=\langle n+1 \rangle  &\mapsto  \langle n \rangle \\
\sigma^n_i &\mapsto  \partial_i^{n-1},\, i=0,\ldots,n-1\\
\partial^n_i &\mapsto  \sigma_{i-1}^{n-1},\, i=1,\ldots,n-1
\end{align}
The tensor product on $\Delta_{\bot\top}$ is inherited from the ordinal sum under the isomorphism (\ref{eq:isoDelta}), and has the interpretation of path concatination;
\begin{align}
\xi&:[n]\rightarrow[m]\\
\xi'&:[n']\rightarrow[m']
\end{align}
concatinate to
\begin{align}
\xi+\xi':[n+n']&\rightarrow[m+m']\\
i&\mapsto 
\begin{cases}
\xi(i),\mathrm{ if }i\leq n\\
\xi'(i-n),\mathrm{ otherwise.}
\end{cases}
\end{align}
In particular, every such 1-cell $\xi$ can be decomposed
\begin{align}\label{eq:pathDecomp}
\xi=\sum_{i=1}^n !:[1]\rightarrow[\xi_i],\mathrm{ with } \sum_{i=1}^n \xi_i=m.
\end{align}
The image of $\xi$ under the isomorphism is an order preserving function that takes $\xi_i$ points in $ \langle m \rangle $ to $i\in  \langle n \rangle $. An example of the isomorphism, for $n=2$ and $m=3$ can be visualized as
\begin{equation}
\begin{tikzpicture}[baseline=(current  bounding  box.center)]
\def \strx {0.4};
\def \stry {1};
{
\node (n3) at (-5 * \strx ,1) {$[3]$};
\node (n2) at (-5 * \strx ,-1) {$[2]$};
\path[<-,font=\scriptsize,>=angle 90,line width=1pt]
(n3) edge node[left] {$\xi$} (n2);
}
{
\node (m3) at (5 * \strx ,1) {$\langle 3 \rangle $};
\node (m2) at (5 * \strx ,-1) {$\langle 2 \rangle $};
\path[->,font=\scriptsize,>=angle 90]
(m3) edge node[right] {$\tilde\xi$} (m2);
}

\node (u0) at (-3 * \strx ,1) {$0$};
\node (u1) at (-1 * \strx ,1 * \stry) {$1$};
\node (u2) at (1 * \strx ,1 * \stry) {$2$};
\node (u3) at (3 * \strx ,1 * \stry) {$3$};

\node (d0) at (-2* \strx ,-1 * \stry) {$0$};
\node (d1) at (0 * \strx ,-1 * \stry) {$1$};
\node (d2) at (2 * \strx ,-1 * \stry) {$2$};

\path[->,font=\scriptsize,>=angle 90]
(u0) edge node[below] (au1) {} (u1);
\path[->,font=\scriptsize,>=angle 90]
(u1) edge node[below] (au2) {} (u2);
\path[->,font=\scriptsize,>=angle 90]
(u2) edge node[below] (au3) {} (u3);
\path[->,font=\scriptsize,>=angle 90]
(d0) edge node[above] (ad1) {} (d1);
\path[->,font=\scriptsize,>=angle 90]
(d1) edge node[above] (ad2) {} (d2);

{
\path[|->,font=\scriptsize,>=angle 90,line width=1pt]
(d0) edge (u0);
\path[|->,font=\scriptsize,>=angle 90,line width=1pt]
(d1) edge (u3);
\path[|->,font=\scriptsize,>=angle 90,line width=1pt]
(d2) edge (u3);
}

{
\path[|->,font=\scriptsize,>=angle 90]
(au1) edge (ad1);
\path[|->,font=\scriptsize,>=angle 90]
(au2) edge (ad1);
\path[|->,font=\scriptsize,>=angle 90]
(au3) edge (ad1);
}

\end{tikzpicture}
\end{equation}

The embedding $\Delta_{\bot\top}\hookrightarrow\Delta_a$ is a monoidal functor with comparison maps representing
\begin{align}
 \langle 0 \rangle &\xrightarrow{\partial_0^0} \langle 1 \rangle  = [0] \label{eq:compDeltaId}\\
{[}n]\oplus[n']=\langle n+n'+2 \rangle &\xrightarrow{z_{n,n'}:=\sigma_{n}^{n+n'+1}}\langle n+n'+1 \rangle =[n]+[n']\label{eq:compDeltaComp}
\end{align}

There is a functor
\begin{align}
\Delta_{\bot\top}^{op} &\xrightarrow{L} \Delta_{a}\\
{[}n{]}=\langle n+1 \rangle  &\mapsto  \langle n+1 \rangle \\
\sigma^n_i &\mapsto  \partial_{i+1}^n,\, i=0,\ldots,n-1\\
\partial^n_i &\mapsto  \sigma_i^n,\, i=1,\ldots,n-1
\end{align}
assigning to each 1-cell in $\Delta_{\bot\top}$ its left adjoint (Galois connection) in $\Delta_{a}$. Explicitly, for $\xi:[n]\rightarrow [m]$,
\begin{align}
L(\xi):\langle m+1 \rangle &\rightarrow \langle n+1 \rangle \\
i&\mapsto  \min\{j|i\leq \xi(j)\}.
\end{align}
The functor $L$ is oplax monoidal, with the same comparison maps (\ref{eq:compDeltaId})-(\ref{eq:compDeltaComp}), but the naturality holds up to a 2-cell
\begin{equation}\label{eq:Loplax}
L(\xi+\xi')\circ z_{m,m'} \Rightarrow z_{n,n'}\circ (L\xi\oplus L\xi').
\end{equation}
Dually, there is a lax monoidal functor $\Delta_{\bot\top}^{op} \xrightarrow{R} \Delta_{a}$ assigning right adjoints, with a 2-cell
\begin{equation}\label{eq:Rlax}
R(\xi+\xi')\circ z_{m,m'}\Leftarrow z_{n,n'}\circ (R\xi\oplus R\xi').
\end{equation}

The free 2-category containing a monad \cite{Lawvere1969} is obtained as the suspension of the monoidal category of intervals,
\begin{equation}
\mathrm{FM}:=\Sigma \Delta_{\bot\top}.
\end{equation}

\subsection{Shuffles - free distributive law}

A shuffle of $\langle n \rangle $ and $\langle m \rangle $ in $\Delta_a$ is defined to be a pair of complement embeddings $\langle n \rangle \rightarrow\langle n+m \rangle \leftarrow\langle m \rangle $. Shuffles in $\Delta_{\bot\top}$ are inherited via the isomorphism (\ref{eq:isoDelta}) and have the following explicit description:
\begin{align}
[n]\xleftarrow{r} [n+m] \xrightarrow{s} [m]
\end{align} 
with the constraint 
\begin{equation}\label{eq:rscomp}
r_i+s_i=1\,.
\end{equation}
The numbers $r_i$ and $s_i$ are lengths (either 0 or 1 in this case) of the image of the $i^\mathrm{th}$ subinterval of $[n+m]$, as in (\ref{eq:pathDecomp}). The condition (\ref{eq:rscomp}) states that each subinterval maps to an interval of length $1$ either in $[n]$ or in $[m]$.

An equivalent description of a shuffle is given by a relation of ``appearing before in the shuffle''
\begin{equation}
\langle m \rangle^\mathrm{op} \times \langle n \rangle
\xrightarrow{l} \langle 2 \rangle\,.
\end{equation}
The same relation can be interpreted as a shuffle of segments $[n]$ and $[m]$, for example
\begin{equation}\label{diag:shufMat}
\begin{tikzpicture}[baseline=(current  bounding  box.center)]
\def \str {0.8}
\node (A0) at (\str * -0.5,\str * 0) {$[2]\backslash [3]$};
\node (C1) at (\str * 1,\str * 0) {$0$};
\node (C2) at (\str * 3,\str * 0) {$1$};
\node (C3) at (\str * 5,\str * 0) {$2$};
\node (C4) at (\str * 7,\str * 0) {$3$};

\node (D1) at (-0.5 * \str,-1 * \str) {$0$};
\node (D2) at (-0.5 * \str,-2 * \str) {$1$};
\node (D3) at (-0.5 * \str,-3 * \str) {$2$};

\node (A11) at (2 * \str,-1.5 * \str) {$0$};
\node (A12) at (4 * \str,-1.5 * \str) {$1$};
\node (A13) at (6 * \str,-1.5 * \str) {$1$};
\node (A11) at (2 * \str,-2.5 * \str) {$0$};
\node (A11) at (4 * \str,-2.5 * \str) {$0$};
\node (A11) at (6 * \str,-2.5 * \str) {$0$};

\path[->,font=\scriptsize,>=angle 90]
(C1) edge node[above] {} (C2)
(C2) edge node[above] {} (C3)
(C3) edge node[above] {} (C4);

\path[->,font=\scriptsize,>=angle 90]
(D1) edge node[left] {} (D2)
(D2) edge node[left] {} (D3);

\node (C1D1) at (1 * \str,-1 * \str) {};
\node (C2D1) at (3 * \str,-1 * \str) {};
\node (C2D2) at (3 * \str,-2 * \str) {};
\node (C3D2) at (5 * \str,-2 * \str) {};
\node (C4D2) at (7 * \str,-2 * \str) {};
\node (C4D3) at (7 * \str,-3 * \str) {};

\path[->,font=\scriptsize,>=angle 90]
(C1D1) edge node[above] {} (C2D1)
(C2D1) edge node[left] {} (C2D2)
(C2D2) edge node[above] {} (C3D2)
(C3D2) edge node[above] {} (C4D2)
(C4D2) edge node[left] {} (C4D3);
\end{tikzpicture}
\end{equation}

A shuffle morphism $(\xi,\rho):(n,m,s,r)\rightarrow(\bar n,\bar m,\bar s,\bar r)$ consists of 1-cells $\xi:[\bar n]\rightarrow [n]$ and $\rho:[\bar m]\rightarrow [m]$ in $\Delta_{\bot\top}$, such that the following 2-cell in $\Delta_a$ exists
\begin{equation}\label{eq:shufmornat}
Lr\circ\xi\circ \bar r \Rightarrow Rs\circ\rho\circ\bar s\;.
\end{equation}
When $\xi=1_{[n]}$ and $\rho=1_{[m]}$, the condition (\ref{eq:shufmornat}) is equivalent to the fact that the induced relations $l,\bar l:\langle m \rangle^\mathrm{op} \times \langle n \rangle
\xrightarrow{} \langle 2 \rangle$ satisfy $l\leq \bar l$, or that the $\bar l$ path in the table (\ref{diag:shufMat}) appears to the down-left of the $l$ path.

Shuffles and their morphisms form a category $\textbf{Shuff}$ with the identity morphism $(1_{[n]},1_{[m]})$ and composition $(\xi\circ \bar \xi,\rho\circ \bar \rho)$ for which the condition (\ref{eq:shufmornat}) is obtained by pasting
\begin{equation}
\begin{tikzpicture}[baseline=(current  bounding  box.center)]
\def \strx {4.0}
\def \stry {1.5}
\def \gmoff {0.7}
\node (C1) at (-\strx,0) {$[\bar{\bar n}+\bar{\bar m}]$};
\node (C2) at (0,0) {$[\bar{ n}+\bar{ m}]$};
\node (C3) at (\strx,0) {$[n+m]$};
\node (U1) at (-0.75 * \strx,\stry) {$[\bar{\bar n}]$};
\node (U2) at (-0.25 * \strx,\stry) {$[\bar{n}]$};
\node (U3) at (0.25 * \strx,\stry) {$[\bar{n}]$};
\node (U4) at (0.75 * \strx,\stry) {$[n]$};
\node (D1) at (-0.75 * \strx,-\stry) {$[\bar{\bar m}]$};
\node (D2) at (-0.25 * \strx,-\stry) {$[\bar{m}]$};
\node (D3) at (0.25 * \strx,-\stry) {$[\bar{m}]$};
\node (D4) at (0.75 * \strx,-\stry) {$[m]$};
\path[->,font=\scriptsize,>=angle 90]
(C1) edge [bend left] node[above left] {$\bar{\bar r}$} (U1);
\path[->,font=\scriptsize,>=angle 90]
(C1) edge [bend right] node[below left] {$\bar{ \bar s}$} (D1);

\path[->,font=\scriptsize,>=angle 90]
(U1) edge node[above] {$\bar{\xi}$} (U2);
\path[->,font=\scriptsize,>=angle 90]
(D1) edge node[below] {$\bar{\rho}$} (D2);

\path[->,font=\scriptsize,>=angle 90]
(U2) edge node[left] {$L\bar{r}$} (C2);
\path[->,font=\scriptsize,>=angle 90]
(D2) edge node[left] {$R\bar{s}$} (C2);

\path[->,font=\scriptsize,>=angle 90]
(U2) edge node[above] {$1$} (U3);
\path[->,font=\scriptsize,>=angle 90]
(D2) edge node[below] {$1$} (D3);

\path[->,font=\scriptsize,>=angle 90]
(C2) edge node[right] {$\bar{r}$} (U3);
\path[->,font=\scriptsize,>=angle 90]
(C2) edge node[right] {$\bar{s}$} (D3);

\path[->,font=\scriptsize,>=angle 90]
(U3) edge node[above] {$\xi$} (U4);
\path[->,font=\scriptsize,>=angle 90]
(D3) edge node[below] {$\rho$} (D4);

\path[->,font=\scriptsize,>=angle 90]
(U4) edge [bend left] node[above right] {$Lr$} (C3);
\path[->,font=\scriptsize,>=angle 90]
(D4) edge [bend right] node[below right] {$Rs$} (C3);

\node at (-0.5 * \strx,0) {\scriptsize {$\Downarrow$}};
\node at (0.5 * \strx,0) {\scriptsize {$\Downarrow$}};
\node at (0,\stry * \gmoff) {\scriptsize $\eta$ {$\Downarrow$}};
\node at (0,-\stry * \gmoff) {\scriptsize $\epsilon$ {$\Downarrow$}};
\end{tikzpicture}
\end{equation}

$\textbf{Shuff}$ inherits a tensor product from $\Delta_{\bot\top}$ which (algebraically) follows from
\begin{align}
\hspace{-3mm}
L(r+r')\circ(\xi+\xi')\circ(\bar r+\bar r')\circ z
&\overset{(\ref{eq:compDeltaComp})}{=} 
L(r+r')\circ z \circ (\xi\oplus\xi')\circ(\bar r\oplus\bar r')\\
&\overset{(\ref{eq:Loplax})}{\Rightarrow} 
z \circ(Lr\oplus Lr')\circ (\xi\oplus\xi')\circ(\bar r\oplus\bar r')\\
&\overset{(\ref{eq:shufmornat})}{\Rightarrow}
z \circ(Rs\oplus Rs')\circ  (\rho\oplus\rho')\circ(\bar s\oplus\bar s')\\
&\overset{(\ref{eq:Rlax})}{\Rightarrow}
R(s+s')\circ  z  \circ (\rho\oplus\rho')\circ(\bar s\oplus\bar s')\\
&\overset{(\ref{eq:compDeltaComp})}{=}
R(s+s')\circ (\rho+\rho')\circ(\bar s+\bar s')\circ z 
\end{align}
but can also be seen as ``direct summing''\footnote{As one would direct sum $k$-matrices between finite-dimensional $k$-vector spaces}
the relation tables, for example the shuffle (\ref{diag:shufMat}) can be interpreted as $([2]\xleftarrow{\sigma^3_1} [3] \xrightarrow{\sigma^2_0\circ \sigma^3_2} [1])+([1]\xleftarrow{\sigma^2_1} [2] \xrightarrow{\sigma^2_0} [1])$.

The free 2-category containing a distributive law is obtained as the suspension of the monoidal category of shuffles,
\begin{equation}
	\mathrm{FDL}:=\Sigma \textbf{Shuff}.
\end{equation}

\subsection{Mixed shuffle morphisms - free mixed distributive law}

The category of mixed shuffles $\textbf{MShuff}$ can be obtained by slightly modifying the construction of $\textbf{Shuff}$; the $\rho$ component of the mixed shuffle morphism has the opposite direction $\rho:[m]\rightarrow[\bar m]$, and the existence condition (\ref{eq:shufmornat}) becomes
\begin{align}
	Lr\circ\xi\circ\bar r\Rightarrow Rs\circ R\rho\circ\bar s.
\end{align}

The 2-category containing a free mixed distributive law (FMDL) is obtained as the suspension of the monoidal category of mixed shuffles,
\begin{equation}
\mathrm{FMDL}:=\Sigma \textbf{MShuff}.
\end{equation}

\bibliographystyle{acm}

\end{document}